\documentclass[smallextended]{svjour3}       

\smartqed  
\usepackage{graphicx}

\usepackage{ifthen}
\usepackage{fancyhdr}
\usepackage[us,12hr]{datetime} 
\usepackage{exscale}
\usepackage{tabularx}
\usepackage{latexsym}
\usepackage{amsmath,amssymb,amstext,amsthm} 
\usepackage{graphicx,color,epsfig}
\usepackage[table]{xcolor}
\usepackage{graphicx}
\usepackage{fullpage}        
\usepackage{float}
\usepackage{verbatim}
\usepackage{enumerate}
\usepackage[bookmarks,pagebackref,
    pdfpagelabels=true, 
    ]{hyperref}
    \numberwithin{equation}{section}
    \numberwithin{figure}{section}
\usepackage{makeidx}
\usepackage{tikz}
\usepackage[final]{pdfpages}
\usepackage{algorithm,algorithmic} 
\usepackage{lineno}
\usepackage[numbers]{natbib}
\numberwithin{table}{section}
\numberwithin{algorithm}{section}

\usepackage[symbol]{footmisc}

\makeindex

\def\hypergeom#1#2#3#4#5{{}_#1 F_{#2}\left({#3\atop#4}; \ #5\right)}

\def\eqref#1{{\normalfont(\ref{#1})}}

\def\eqref#1{{\normalfont(\ref{#1})}}

\newtheorem{lem}[theorem]{Lemma}
\newtheorem{thm}[theorem]{Theorem}
\newtheorem{cor}[theorem]{Corollary}
\newtheorem{rem}[theorem]{Remark}

\newcommand{\A}{{\mathcal A}}

\newcommand{\bbm}{\begin{bmatrix}}
\newcommand{\ebm}{\end{bmatrix}}
\newcommand{\bem}{\begin{pmatrix}}
\newcommand{\eem}{\end{pmatrix}}
\newcommand{\beq}{\begin{equation}}
\newcommand{\beqs}{\begin{equation*}}
\newcommand{\bet}{\begin{table}}
\newcommand{\eeq}{\end{equation}}
\newcommand{\eeqs}{\end{equation*}}
\newcommand{\beqr}{\begin{eqnarray}}

\newcommand{\nc}{\newcommand}
\nc{\arrow}{{\rm arrow\,}}
\nc{\Arrow}{{\rm Arrow\,}}
\nc{\BoDiag}{{\rm B^0Diag\,}}
\nc{\bodiag}{{\rm b^0diag\,}}

\nc{\Mm}{{\mathcal M}^{m} }
\nc{\Mmn}{{\mathcal M}^{mn} }
\nc{\Mnr}{{\mathcal M}_{nr} }
\nc{\Mnmr}{{\mathcal M}_{(n-1)r} }
\nc{\kwqqp}{Q{$^2$}P\,}
\nc{\kwqqps}{Q{$^2$}Ps}

\nc{\notinaho}{(X,S)\in \overline{AHO}(\A)}
\nc{\inaho}{(X,S)\in AHO(\A)}

\newcommand{\bea}{\begin{eqnarray}}%
\newcommand{\eea}{\end{eqnarray}}%
\newcommand{\beas}{\begin{eqnarray*}}%
\newcommand{\eeas}{\end{eqnarray*}}%
{}

\newcommand{\Hnp}[1][]{\,\mathbb{H}_+^{\ifthenelse{\equal{#1}{}}{n}{#1}}}
\newcommand{\Hn}[1][]{\,\mathbb{H}^{\ifthenelse{\equal{#1}{}}{n}{#1}}}
\newcommand{\Dn}[1][]{\,\mathbb{D}^{\ifthenelse{\equal{#1}{}}{n}{#1}}}

\begin{document}

\title{Spectral density functions of bivariable stable polynomials\thanks{Research of HJW is supported by 
Simons Foundation grant 355645 and National Science Foundation grant DMS 2000037.}}


\author{Jeffrey S. Geronimo  \and Hugo J. Woerdeman \and
        Chung Y. Wong 
}


\institute{Jeffrey S. Geronimo\at School of Mathematics\\ Georgia Institute of Technology\\ 225 North Ave\\ Atlanta, GA 30332
              \email{geronimo@math.gatech.edu}           
           \and
           Hugo J. Woerdeman \at
             Department of Mathematics\\ Drexel University\\ 
3141 Chestnut Street\\ Philadelphia, PA 19104 \email{hugo@math.drexel.edu} 
\and 
Chung Y. Wong \at Department of Mathematics\\ County College of Morris\\ 214 Center Grove Rd.\\
Randolph, NJ 07869 \email{cwong@ccm.edu} 
}

\date{Received: date / Accepted: date}

\maketitle

\begin{abstract}
The relationship between a stable multivariable polynomial $p(z)$ and the Fourier coefficients of its spectral density 
function $1/|p(z)|^2$, is further investigated. In this paper we focus on the radial asymptotics of the Fourier coefficients for a specific choice of a two variable polynomial. Hypergeometric functions appear in the analysis, and new results are derived for these as well.
\end{abstract}

\section{Introduction}
\label{sec:intro}

The design of multivariate autoregressive filters in statistical signal processing involves estimating
the autocorrelation coefficients of a stationary process and subsequently computing the coefficients 
of the filter. Mathematically the autoregressive filter is represented by a {\it stable} polynomial $p(z_1, \ldots z_d )$ in $d$
variables; that is, $p(z_1, \ldots , z_d ) \neq 0$ for $(z_1, \ldots ,  z_d ) \in 
\overline{\mathbb D}^d$, where ${\mathbb D} = \{ w \in {\mathbb C} : |w| <1 \}$. We will also use the notation  ${\mathbb T} = \{ w \in {\mathbb C} : |w| =1 \}$ and $\overline{\mathbb D} = {\mathbb D} \cup {\mathbb T}$. 
The corresponding autocorrelation coefficients $(c_k)_{k \in {\mathbb Z}^d}$  are then represented by the Fourier coefficients of the so-called {\it spectral density function} $f(z) = \frac{1}{p(z)\overline{p}(1/z)}$, where for $p(z)=\sum p_k z^k$ we have $\overline{p} (1/z) =\sum \overline{p}_k z^{-k}$ (we use the multivariable notation $z^k = z_1^{k_1} \cdots z_d^{k_d}$ when $z=(z_1, \ldots , z_d )$ and $k=(k_1, \ldots , k_d )$). Thus 
$$c_k=  \frac{1 }{(2\pi)^d} \int_{[0,2\pi]^d} f(e^{i\theta_1}, \ldots 
, e^{i\theta_d}) e^{-i\sum_{r=1}^d \theta_r k_r } d\theta_1 \cdots d\theta_d ,  k=(k_1, \ldots , k_d )\in{\mathbb Z}^d .  $$
Note that $c_k = \overline{c_{-k}}$. 


A polynomial $p(z)=p(z_1, \ldots ,  z_d )$ of degree $(n_1, \ldots , n_d)$ (thus of degree $n_1$ in $z_1$, $n_2$ in $z_2$, etc.) has
$\prod_{i=1}^d (n_i+1)$ coefficients. Ideally one would like to determine $p(z)$ based on the same number of autocorrelation coefficients. In the classical case, when $d=1$, the mathematical foundations go back to Toeplitz and Szeg\"o in the 1910s, while the signal processing application was developed in the 1950s. The connecting equation is the so-called Yule-Walker equation 
\begin{equation}\label{YW} \begin{pmatrix} c_0 & c_{-1} & \cdots & c_{-n} \cr
c_1 & c_0 & \cdots & c_{-n+1} \cr
\vdots & \vdots & \ddots & \vdots \cr
c_n & c_{n-1} & \cdots & c_0
\end{pmatrix}^{-1} \begin{pmatrix} 1 \cr 0 \cr \vdots \cr 0 \end{pmatrix} = \overline{p_0} \begin{pmatrix} p_0 \cr p_1 \cr \vdots \cr p_n \end{pmatrix}  . \end{equation}
A solution $p(z)=\sum_{j=0}^n p_j z^j$ exists exactly when the Toeplitz matrix in \eqref{YW} is positive definite, and in that case the polynomial can be determined via \eqref{YW}. Since $c_k = \overline{c_{-k}}$, one only needs $c_0 , \ldots , c_n$ to determine $p_0 , \ldots , p_n$ (uniquely, when one requires the normalization $p_0>0$). While in the 1970s there were some partial results for the two-variable case (see, e.g., \cite{DGK}), the autoregressive filter problem for $d=2$ was resolved in the 2000s by two of the current authors; see \cite{GWAnnals}. Again there is a Yule-Walker type equation, but now the number of autocorrelation coefficients needed to build the two-variable Toeplitz matrix is asymptotically twice the number of coefficients of the polynomial. Thus there is a mismatch, which means that the autocorrelation coefficients must satisfy additional constraints. In \cite{GWAnnals} it was shown that these additional constraints appear in the form of a low rank submatrix of the two-variable Toeplitz matrix. As a byproduct the authors obtained in \cite{GWAnnals} an expression for the Fourier coefficients $c_{k_1, k_2}$ in the region where $k_1k_2 \le 0$. Due to the stability of $p$, this expression just involves the common zeros of $p(z_1,z_2)$ and it reverse $\overleftarrow{p} (z_1, z_2  ):= z_1^{n_1} z_2^{n_2} \overline{p} (1/z_1, 1/z_2 )$, which the authors termed the {\it intersecting zeros} of $p$ \cite{GWAnnals, GWIEEE}. Using this expression it is easy to analyze asymptotics of $c_{k_1, k_2}$ in the quadrants where $k_1k_2 \le 0$. The asymptotics in the other quadrants was not addressed. This gave the impetus for the current paper. While, as we will see, the asymptotics question can be addressed in great generality by analytic combinatorics theory (see \cite{PW}), we are also motivated to understand the direct connection between the coefficients $p_k$ of the polynomial and the autocorrelation coefficients $c_k$. We therefore decided to analyze in depth the correspondence for the special case when $p(z_1, z_2) = 1-\frac{z_1+z_2}{r}$, where $r>2$, in part as this polynomial lends itself to generalization to three or more variables where no general theory exists yet. As it turned out this two variable example leads to formulas for the autocorrelation coefficients in terms of hypergeometric functions. Due to this we will along the way also develop some new identities and asymptotic results for hypergeometric functions. 

The paper is organized as follows. In Section 2 we determine the Fourier coefficients of the spectral density function of $1-\frac{z_1+z_2}{r}$, along with new identities for hypergeometric functions. In Section 3 we study the radial asymptotics of these Fourier coefficients, which involves the asymptotics of certain hypergeometric functions with large parameters. In Section 4 we determine the orthogonal polynomials associated with the measure on the bitorus with weight given by the spectral density function. 

\section{Determining the Fourier Coefficients}
\label{sec:Fourier}

As we will see in Theorem \ref{coeffi}, hypergeometric functions come up in a natural way when computing the Fourier coefficients $c_{k_1,k_2}$ of the spectral density function associated with $ p(z_1,z_2) = 1 - \frac{z_1 +z_2}{r} $.
Recall that the hypergeometric functions $\ _2 F_1$ and $\ _3 F_2$ are defined for $|x|<1$ via the power series
$$ \hypergeom{2}{1}{a,b}{d}{x} =\sum_{n=0}^\infty \frac{(a)_n (b)_n}{(d)_n} \frac{x^n}{n!} , \ \hypergeom{3}{2}{a,b,c}{d,e}{x} =\sum_{n=0}^\infty \frac{(a)_n (b)_n(c)_n}{(d)_n(e)_n} \frac{x^n}{n!}. $$ Here the Pochhammer function $(q)_n$ is defined by
$$ (q)_n = \begin{cases} 1 , & n=0; \\ q(q+1)\cdots(q+n-1) , & {\rm otherwise} . \end{cases}$$
Alternatively, $(q)_n$ is also referred to as the rising factorial. 
We refer to the parameters $a,b,c$ as {\it numerator parameters} and to $d,e$ as {\it denominator parameters}. Note that when $c=e$, we have $\hypergeom{3}{2}{a,b,c}{d,c}{x} =  \hypergeom{2}{1}{a,b}{d}{x}$.
One of the formulas we will be using repeatedly is the {\it Chu-Vandermonde identity}, which states that 
\begin{equation}\label{chu} \sum_{s=0}^n \frac{(-n)_s (b)_s}{s! \ (d)_s } = \hypergeom{2}{1}{-n,b}{d}{1} = \frac{ (d-b)_n}{(d)_n}. \end{equation}
Other useful identities involving hypergeometric functions may, for instance, be found in \cite{AAR, BW, GR, Olver}. We now summarize those that we will be using.

\begin{lem}
The following relations hold:
\begin{equation}\label{cont1}
b\ \hypergeom{3}{2}{a,a_2,a_3}{b,b_2}{x}-a\ \hypergeom{3}{2}{a+1,a_2,a_3}{b+1,b_2}{x}+(a-b)\hypergeom{3}{2}{a,a_2,a_3}{b+1,b_2}{x}=0,
\end{equation}
\begin{equation}\label{cont2}
bcx\ \hypergeom{3}{2}{a+1,b+1,c+1}{d+1,e+1}{x}+de\left(\hypergeom{3}{2}{a,b,c}{d,e}{x}-\hypergeom{3}{2}{a+1,b,c}{d,e}{x}\right)=0,
\end{equation}
\begin{equation}\label{cont3}
(b-c)\hypergeom{2}{1}{a,b-1,}{c}{x}-a(x-1)\hypergeom{2}{1}{a+1,b}{c}{x}+(c-a-b)\hypergeom{2}{1}{a,b}{c}{x}=0,
\end{equation}
and 
\begin{equation}\label{Pfaff}
\hypergeom{2}{1}{a,b}{c}{x}=(1-x)^{-a} \hypergeom{2}{1}{a,c-b}{c}{\frac{x}{x-1}}.
\end{equation}
\end{lem}
\begin{proof} 
The first three can be checked directly by comparing the coefficients of $x^n$ on both sides in each of the equations. These contiguous relations can also be found on the Wolfram webpages \cite{wolfram} for $_2F_1$'s and $_3F_2$'s. The fourth equation is known as Pfaff's transformation formula; see, for instance, \cite[Theorem 2.2.5]{AAR}. 
\end{proof}

Our first main result gives an expression for the Fourier coefficients of the spectral density function of $1-\frac{z_1+z_2}{r}$.

\begin{theorem}\label{coeffi} 
Let $ p(z_1,z_2) = 1 - \frac{z_1 +z_2}{r} $
with $r>2$, and let $ c_{k_1,k_2}$ denote the Fourier coefficients of its spectral density function.
Then we have
$$ c_{k_1,k_2} = \frac{1}{\sqrt{1-\frac{4}{r^2}}} \left(\frac{r}{2}-\sqrt{\frac{r^2}{4}-1}\right)^{|k_1|+|k_2|}, \  k_1 k_2 \le 0, $$
and 
$$ c_{k_1,k_2} = \frac{{|k_1|+|k_2| \choose |k_1|}}{r^{|k_1|+|k_2|}}\ \hypergeom{3}{2}{1,\frac{|k_1|+|k_2|}{2}+1,\frac{|k_1|+|k_2|+1}{2}}{|k_1|+1,|k_2|+1}{\frac{4}{r^2}}, \ k_1 k_2 >0. $$
\end{theorem}

\begin{proof}
Observe that for $|z_1+z_2|<r$,
\begin{align}\label{2dinv}
\frac{1}{p(z_1,z_2)}&=\sum_{i=0}^{\infty}\left(\frac{z_1+z_2}{r}\right)^i\nonumber\\&=\sum_{i=0}^{\infty}\frac{1}{r^i}\sum_{j=0}^i\frac{i!}{j!(i-j)!}z_1^{i-j}z_2^{j}\nonumber\
\cr &=\sum_{j=0}^{\infty}\sum_{i=j}^\infty\frac{1}{r^i}\frac{i!}{j!(i-j)!}z_1^{i-j}z_2^{j}\nonumber\\&=\sum_{j=0}^{\infty}\sum_{i=0}^{\infty}\frac{1}{r^{i+j}}\frac{(i+j)!}{j!i!}z_1^{i}z_2^{j}.
\end{align}
It should be noted that the series above as well as subsequent ones are absolutely convergent on a domain containing ${\mathbb T}^2 = \{ (z_1,z_2) : |z_1|=|z_2|=1 \}$, which thus allows us to freely rearrange the terms as desired.
Therefore, on a domain containing ${\mathbb T}^2$, 
\begin{align}\label{3dsp}
f(z_1,z_2)&=\frac{1}{p(z_1,z_2)\bar p(1/z_1,1/z_2)}\nonumber\\&=\sum_{i_1=0}^{\infty}\sum_{i_2=0}^{\infty}\sum_{j_1=0}^{\infty}\sum_{j_2=0}^{\infty}\frac{(i_1+j_1)!}{j_1!i_1!}\frac{(i_2+j_2)!}{j_2!i_2!}\frac{z_1^{i_1-i_2}z_2^{j_1-j_2}}{r^{i_1+i_2+j_1+j_2}}\nonumber\\&=\sum_{k_1=-\infty}^{\infty}\sum_{k_2=-\infty}^{\infty}c_{k_1,k_2}z_1^{k_1}z_2^{k_2}.
\end{align}
The above sum on $c_{k_1,k_2}$ can be divided into four regions depending on the signs of $k_1$ and $k_2$. However since $\overline{c_{k_1,k_2}} = c_{-k_1, -k_2}$ we may restrict ourselves to the case when $k_1\ge 0$.

First let $k_1\ge0$ and $k_2\ge0$. Matching coefficients for $k_2=j_1-j_2$ and $k_1=i_1-i_2$ with $j_1=k_2+j_2$ and $i_1=k_1+i_2$ then setting $i=i_2+j_2$ and $j=j_2$ yields,
$$
c_{k_1,k_2}=\sum_{i=0}^{\infty}\frac{1}{r^{2i+k_1+k_2}}\sum_{j=0}^{i}\frac{i!}{j!(i-j)!}\frac{(k_1+k_2+i)!}{(k_1+i-j)!\ (k_2+j)!}.
$$
Since 
\begin{equation}\label{kkdent}
(k_2+j)!=(k_2+1)_j k_2!,
\end{equation}
\begin{equation}\label{iident}
\frac{i!}{(i-j)!}=(-1)^j(-i)_j,
\end{equation}
and
\begin{equation}\label{kident}
(k_1+i-j)!=\frac{(k_1+i)!}{(k_1+i-j+1)\cdots(k_1+i)}=(-1)^j\frac{(k_1+i)!}{(-k_1-i)_j},
\end{equation}
the sum on $j$ can be written as
\begin{align*}
\sum_{j=0}^i\frac{i!(k_1+k_2+i)!}{j!(i-j)!(k_1+i-j)!(k_2+j)!}&=\frac{(k_1+k_2+i)!}{k_2!(k_1+i)!}\sum_{j=0}^i\frac{(-i)_j(-k_1-i)_j}{j!(k_2+1)_j}\\&=\frac{(k_1+k_2+i)!(k_1+k_2+i+1)_i}{k_2!(k_1+i)!(k_2+1)_i}, 
\end{align*}
where the Chu-Vandermonde identity \eqref{chu} has been used to obtain the last equality. With the identity $\frac{(k_1+k_2+i)!}{(k_1+i)!}=\frac{(k_1+k_2)!(k_1+k_2+1)_i}{k_1!(k_1+1)_i}$ we find 

\begin{equation}\label{c000}
c_{k_1,k_2}=\frac{(k_1+k_2)!}{k_1!k_2!}\sum_{i=0}^{\infty}\frac{1}{r^{2i+k_1+k_2}}\frac{(k_1+k_2+1)_i(k_1+i+k_2+1)_i}{(k_1+1)_i(k_2+1)_i}.
\end{equation}
Substitution of the identity
\begin{equation}\label{k1k2ident}
(k_1+k_2+i+1)_i=4^i\frac{\left(\frac{k_1+k_2}{2}+1\right)_i\left(\frac{k_1+k_2+1}{2}\right)_i}{(k_1+k_2+1)_i }
\end{equation}
yields
\begin{align}\label{ck1k2}
c_{k_1,k_2}&=\frac{(k_1+k_2)!}{k_1!k_2!}\sum_{i=0}^{\infty}\frac{4^i}{r^{2i+k_1+k_2}}\frac{\left(\frac{k_1+k_2}{2}+1\right)_i\left(\frac{k_1+k_2+1}{2}\right)_i}{(k_2+1)_i(k_1+1)_i}\nonumber\\&=\frac{(k_1+k_2)!}{k_1!k_2!}\frac{1}{r^{k_1+k_2}}\hypergeom{3}{2}{1,\frac{k_1+k_2}{2}+1,\frac{k_1+k_2+1}{2}}{k_1+1,k_2+1}{\left(\frac{2}{r}\right)^2}.
\end{align}
This proves the case when $k_1\ge0$ and $k_2\ge0$. 

Next we compute $c_{k_1, -k_2}$, where $k_1,k_2 \ge 0$. 
%
%
Using \eqref{3dsp} with $i_1=k_1+i_2$ and $j_1=j_2-k_2$, the spectral density function restricted to this region can be written as
$$
\sum_{k_1=0}^{\infty}\sum_{k_2=0}^{\infty}c_{k_1,-k_2}z_1^{k_1}z_2^{-k_2}= \sum_{k_1=0}^{\infty}\sum_{j_2=0}^{\infty}\sum_{k_2=0}^{j_2}\sum_{i_2=0}^{\infty}\frac{(k_1+i_2+j_2-k_2)!}{(j_2-k_2)!(i_2+k_1)!}\frac{(i_2+j_2)!}{j_2!i_2!}\frac{z_1^{k_1} z_2^{-k_2}}{r^{k_1-k_2+2i_2+2j_2}}.
$$
Matching coefficients on both sides, we have with
$i=i_2+j_2-k_2$ and $j=j_2-k_2$,
\begin{equation}\label{ck2neg}
c_{k_1,-k_2}=\sum_{i=0}^{\infty}\frac{1}{r^{2i+k_1+k_2}}\sum_{j=0}^{i}\frac{(i+k_1)!}{j!(i+k_1-j)!}\frac{(i+k_2)!}{(j+k_2)!(i-j)!}.
\end{equation}
 Substitution of the identities \eqref{kkdent}, \eqref{iident}, and \eqref{kident} into the sum on $j$ gives,
\begin{align*}
\sum_{j=0}^i\frac{(i+k_1)!}{(i+k_1-j)!}\frac{1}{j!(i-j)!}\frac{(i+k_2)!}{(j+k_2)!}&=\frac{(i+k_2)!}{i!k_2!}\sum_{j=0}^i\frac{(-i)_j(-k_1-i)_j}{j!(k_2+1)_j}\\&=\frac{(i+k_2)!}{i!k_2!}\frac{(k_1+k_2+i+1)_i}{(k_2+1)_i}\\&=\frac{1}{i!}4^i\frac{\left(\frac{k_1+k_2}{2}+1\right)_i\left(\frac{k_1+k_2+1}{2}\right)_i}{(k_1+k_2+1)_i} .
\end{align*}
The second to last equality is obtained from the Chu-Vandermonde formula \eqref{chu} and the last equality uses \eqref{k1k2ident}. Substitution of the above result into the equation for $c_{k_1,-k_2}$ yields

\begin{equation}\label{ck2negf}
c_{k_1,-k_2}=\frac{1}{r^{k_1+k_2}} \ \hypergeom{2}{1}{\frac{k_1+k_2+1}{2},\frac{k_1+k_2}{2}+1}{k_1+k_2+1}{\frac{4}{r^2}}.
\end{equation}
In general we have the following formula (see, e.g., \cite[Section 8.9]{BW})
\begin{equation}\label{quadra} 
\hypergeom{2}{1}{a, b}{ a + b - \frac12}{ x} = \frac{1}{\sqrt{1 - x}} \left( \frac{1 + \sqrt{1 - x}}{2}\right)^{1 - 2 a} \hypergeom{2}{1}{2 a - 1, a - b + \frac12}{ a + b - \frac12 }{ \frac{\sqrt{1 - x} - 1}{\sqrt{1 - x} + 1} }.
\end{equation}
Applying this to \eqref{ck2negf} we get
\begin{equation} c_{k_1,-k_2}  = \frac{1}{r^{k_1+k_2}}\frac{1}{\sqrt{1 - \frac{4}{r^2}}} \left( \frac{1 + \sqrt{1 - \frac{4}{r^2}}}{2}\right)^{-k_1-k_2} \hypergeom{2}{1}{k_1+k_2, 0}{ k_1+k_2+1 }{\frac{\sqrt{1 - \frac{4}{r^2}} - 1}{\sqrt{1 - \frac{4}{r^2}} + 1} } = \nonumber \end{equation}
\begin{equation} \frac{1}{r^{k_1+k_2}}\frac{1}{\sqrt{1 - \frac{4}{r^2}}} \left( \frac{1 + \sqrt{1 - \frac{4}{r^2}}}{2}\right)^{-k_1-k_2} = \frac{1}{\sqrt{1 - \frac{4}{r^2}}}\left( \frac{r}{2} -\sqrt{\frac{r^2}{4}-1} \right)^{k_1+k_2}, \nonumber \end{equation}
where we have used
$$ \frac{2}{r}\frac{1}{1+\sqrt{1-\frac{4}{r^2}}} =  \frac{2}{r}\ \frac{1-\sqrt{1-\frac{4}{r^2}}}{1-\left(1-\frac{4}{r^2}\right)}=\frac{r}{2} -\sqrt{\frac{r^2}{4}-1} .$$
\end{proof}

%

Below we give alternate expressions for ${}_3F_2$. These formulas, which do not seem to appear in the literature, were inspired by the quadratic transformation formulas for ${}_3F_2$, such as \cite[Equation 3.1.15]{AAR}.
The hope was that with the use of such a quadratic transformation a numerator parameter of the resulting hypergeometric function would match with a denominator parameter, and thus reduce to a ${}_2F_1$. As we will see, this hope indeed materializes after some prior and post computations.

\begin{theorem}\label{3f22f1} The following identities hold.
\begin{itemize}
\item[1.]
\begin{equation}\label{ck1ge}
\hypergeom{3}{2}{1,\frac{k_1+k_2}{2}+1,\frac{k_1+k_2+1}{2}}{k_1+1,k_2+1}{x} =  \frac{k_1! k_2!}{(k_1+k_2)!}\frac{2^{k_1+k_2}}{x^{k_2}}\frac{(1+\sqrt{1-x})^{k_2-k_1}}{\sqrt{1-x}} + \end{equation}

$$\frac{k_2}{(k_1+k_2) \sqrt{1-x}} \left[ \hypergeom{2}{1}{1,k_1+k_2}{k_1+1}{ \frac12 -\frac12 \sqrt{1-x}} -\hypergeom{2}{1}{1,k_1+k_2}{k_1+1}{ \frac12 +\frac12 \sqrt{1-x}} \right].  
$$
%
\item[2.]
\begin{align}\label{ck1ge2}
&\hypergeom{3}{2}{1,\frac{k_1+k_2}{2}+1,\frac{k_1+k_2+1}{2}}{k_1+1,k_2+1}{x}=\nonumber\\&  \frac{1}{(k_1+k_2)\sqrt{1-x}}\left[k_2\  \hypergeom{2}{1}{1,k_1+k_2}{k_1+1}{ \frac12 -\frac12 \sqrt{1-x}}+ k_1\ \hypergeom{2}{1}{1,k_1+k_2}{k_2+1}{ \frac12 -\frac12 \sqrt{1-x}} \right].
\end{align}

\end{itemize}
\end{theorem}

\begin{proof} 
Using relation~\eqref{cont1} with $b=k_2+1$, $a=\frac{k_1+k_2+1}{2}$, $a_2=1$, $a_3=\frac{k_1+k_2}{2}+1$ and $b_2=k_1+1$, we find
\begin{align*}
&(k_2+1)\hypergeom{3}{2}{1,\frac{k_1+k_2}{2}+1,\frac{k_1+k_2+1}{2}}{k_1+1,k_2+1}{x} -\frac{k_1+k_2+1}{2}\hypergeom{3}{2}{1,\frac{k_1+k_2}{2}+1,\frac{k_1+k_2+3}{2}}{k_1+1,k_2+2}{x}\nonumber\\&+\frac{k_1-k_2-1}{2}\hypergeom{3}{2}{1,\frac{k_1+k_2}{2}+1,\frac{k_1+k_2+1}{2}}{k_1+1,k_2+2}{x}=0.
\end{align*}
Now use \eqref{cont2} to eliminate $\hypergeom{3}{2}{1,\frac{k_1+k_2}{2}+1,\frac{k_1+k_2+3}{2}}{k_1+1,k_2+2}{x}$ with $a=0$, $b=\frac{k_1+k_2+1}{2}$, $c=\frac{k_1+k_2}{2}$, $d=k_2+1$ and $e=k_1$ to find
\begin{align*}
&\hypergeom{3}{2}{1,\frac{k_1+k_2}{2}+1,\frac{k_1+k_2+1}{2}}{k_1+1,k_2+1}{x}= -\frac{k_1-k_2-1}{2(k_2+1)}\hypergeom{3}{2}{1,\frac{k_1+k_2}{2}+1,\frac{k_1+k_2+1}{2}}{k_1+1,k_2+2}{x}\nonumber\\&-\frac{2k_1}{(k_1+k_2)x}
+\frac{2k_1}{(k_1+k_2)x}\,\hypergeom{3}{2}{1,\frac{k_1+k_2+1}{2},\frac{k_1+k_2}{2}}{k_1,k_2+1}{x}.
\end{align*}
Another application of \eqref{cont1} with $b=k_1$ and $a=\frac{k_1+k_2}{2}$ to eliminate the 2nd hypergeometric function on the right hand side of the above equation yields
\begin{align*}
&\left(1-\frac{1}{x}\right)\hypergeom{3}{2}{1,\frac{k_1+k_2}{2}+1,\frac{k_1+k_2+1}{2}}{k_1+1,k_2+1}{x}= -\frac{k_1-k_2-1}{2(k_2+1)}\hypergeom{3}{2}{1,\frac{k_1+k_2}{2}+1,\frac{k_1+k_2+1}{2}}{k_1+1,k_2+2}{x}\nonumber\\&-\frac{2k_1}{(k_1+k_2)x}
+\frac{k_1-k_2}{(k_1+k_2)x}\,\hypergeom{3}{2}{1,\frac{k_1+k_2+1}{2},\frac{k_1+k_2}{2}}{k_1+1,k_2+1}{x}.
\end{align*}
Both ${}_3F_2$'s on the right hand side of the above equation allow the use of the quadratic identity \cite[equation~(3.1.15)]{AAR} with the change of variables $x \to -\frac{4x}{(1-x)^2}$, $a\to 2a$, $c\to d-b$ and $a-b-1=d$ yields
\begin{align}\label{d3f2}
&\left(1-\frac{1}{x}\right)\hypergeom{3}{2}{1,\frac{k_1+k_2}{2}+1,\frac{k_1+k_2+1}{2}}{k_1+1,k_2+1}{x}= -\frac{2k_1}{(k_1+k_2)x}\nonumber\\&-\frac{k_1-k_2-1}{2(k_2+1)}\left(\frac{2}{x}(1-\sqrt{1-x})\right)^{k_1+k_2+1}\hypergeom{3}{2}{k_1+k_2+1,k_1,k_2+1}{k_1+1,k_2+2}{1-\frac{2}{x}(1-\sqrt{1-x})}\nonumber\\&+\frac{k_1-k_2}{(k_1+k_2)x}\left(\frac{2}{x}(1-\sqrt{1-x})\right)^{k_1+k_2}\hypergeom{3}{2}{k_1+k_2,k_1,k_2}{k_1+1,k_2+1}{1-\frac{2}{x}(1-\sqrt{1-x})}.
\end{align}
Note that the factor of $\frac12$ has been eliminated from the numerator parameters in the hypergeometric functions on the right hand side of the above equation which is why the quadratic transformation is so useful. It is also perhaps worth noting that these hypergeometric functions are easily seen, after interchanging the denominator parameters, as being well poised (see \cite[Definition~3.3.2]{AAR}). Relation \eqref{cont1} with $a=k_1$, $a_2=k_2+1$, $a_3=k_1+k_2+1$, $b=k_2+1$ and $b_2=k_1+1$, can be used to reduce  these functions to  a combination of ${}_2 F_1$'s, as follows for the first hypergeometric function on the left hand side,
\begin{align}\label{prepfaff}
&\hypergeom{3}{2}{k_1+k_2+1,k_1,k_2+1}{k_1+1,k_2+2}{x}\nonumber\\&= \frac{k_1}{k_1-k_2-1}\,\hypergeom{2}{1}{k_1+k_2+1,k_2+1}{k_2+2}{x}-\frac{k_2+1}{k_1-k_2-1}\,\hypergeom{2}{1}{k_1+k_2+1,k_1}{k_1+1}{x}.
\end{align}
Applying the Pfaff transformation \eqref{Pfaff} 
to \eqref{prepfaff}
with $x=1-\frac{2}{x}(1-\sqrt{1-x})$ and $a=k_1+k_2+1$, and using
the fact that $\frac{1-\frac{2}{x}(1-\sqrt{1-x})}{-\frac{2}{x}(1-\sqrt{1-x})}=\frac{1-\sqrt{1-x}}{2}$,
yields
\begin{align*}
&\left(\frac{2(1-\sqrt{1-x})}{x}\right)^{k_1+k_2+1}\hypergeom{3}{2}{k_1+k_2+1,k_1,k_2+1}{k_1+1,k_2+2}{1-\frac{2}{x}(1-\sqrt{1-x})}\nonumber\\&= \frac{k_1}{k_1-k_2-1}\,\hypergeom{2}{1}{1,k_1+k_2+1}{k_2+2}{\frac{1-\sqrt{1-x}}{2}}-\frac{k_2+1}{k_1-k_2-1}\,\hypergeom{2}{1}{1,k_1+k_2+1}{k_1+1}{\frac{1-\sqrt{1-x}}{2}}.
\end{align*}
Substituting the above equation into equation~\eqref{d3f2} and also a similar equation with $k_2\to k_2-1$
in the second hypergeometric function on the right hand side of equation~\eqref{d3f2} gives,
\begin{align}\label{newrelp}
&\left(1-\frac{1}{x}\right)\hypergeom{3}{2}{1,\frac{k_1+k_2}{2}+1,\frac{k_1+k_2+1}{2}}{k_1+1,k_2+1}{x}= -\frac{2k_1}{(k_1+k_2)x}\nonumber\\&-\frac{k_1}{2(k_2+1)}\ \hypergeom{2}{1}{1,k_1+k_2+1}{k_2+2}{\frac{1-\sqrt{1-x}}{2}}+\frac{1}{2}\ \hypergeom{2}{1}{1,k_1+k_2+1}{k_1+1}{\frac{1-\sqrt{1-x}}{2}}
\nonumber\\& +\frac{k_1}{x(k_1+k_2)}\ \hypergeom{2}{1}{1,k_1+k_2}{k_2+1}{\frac{1-\sqrt{1-x}}{2}}-\frac{k_2}{x(k_1+k_2)}\ \hypergeom{2}{1}{1,k_1+k_2}{k_1+1}{\frac{1-\sqrt{1-x}}{2}}.
\end{align}
Assume for the moment that $k_1\ge0$ but not an integer. The relations between hypergeometric functions in $x$ and $1-x$ (see \cite[Corollary 2.3.3]{AAR}) shows that
\begin{align*}
&\frac{k_1}{2(k_2+1)}\ \hypergeom{2}{1}{1,k_1+k_2+1}{k_2+2}{\frac{1-\sqrt{1-x}}{2}}\\&=-\frac{1}{2}\ \hypergeom{2}{1}{1,k_1+k_2+1}{k_1+1}{\frac{1+\sqrt{1-x}}{2}}+\frac{k_1! k_2!}{(k_1+k_2)!}\frac{2^{k_1+k_2}}{x^{k_2+1}}\left(1+\sqrt{1-x}\right)^{-k_1+k_2+1}.
\end{align*}
Now take limits of $k_1$ to a nonnegative integer.
Also
\begin{align}\label{recur3}
&\frac{k_1}{x(k_2+1)}\ \hypergeom{2}{1}{1,k_1+k_2}{k_2+1}{\frac{1-\sqrt{1-x}}{2}}\\&=-\frac{k_2}{x(k_1+k_2)}\ \hypergeom{2}{1}{1,k_1+k_2}{k_1+1}{\frac{1+\sqrt{1-x}}{2}}+\frac{k_1! k_2!}{(k_1+k_2)!}\frac{2^{k_1+k_2}}{x^{k_2+1}}\left(1+\sqrt{1-x}\right)^{-k_1+k_2}.
\end{align}
Substitution of the above two equations into equation~\eqref{newrelp} and simplification yields
\begin{align*}
&\hypergeom{3}{2}{1,\frac{k_1+k_2}{2}+1,\frac{k_1+k_2+1}{2}}{k_1+1,k_2+1}{x}= \frac{2k_1}{(k_1+k_2)(1-x)}+\frac{k_1! k_2!}{(k_1+k_2)!}\frac{2^{k_1+k_2}}{x^{k_2}}\frac{\left(1+\sqrt{1-x}\right)^{-k_1+k_2}}{\sqrt{1-x}}\\&-\frac{x}{2(1-x)}\left(\hypergeom{2}{1}{1,k_1+k_2+1}{k_1+1}{\frac{1-\sqrt{1-x}}{2}}+\hypergeom{2}{1}{1,k_1+k_2+1}{k_1+1}{\frac{1+\sqrt{1-x}}{2}}\right)\\&+\frac{k_2}{(k_1+k_2)(1-x)}\left(\hypergeom{2}{1}{1,k_1+k_2}{k_1+1}{\frac{1-\sqrt{1-x}}{2}}+\hypergeom{2}{1}{1,k_1+k_2}{k_1+1}{\frac{1+\sqrt{1-x}}{2}}\right) .
\end{align*}
Relation~\eqref{cont3} with $a=k_1+k_2$, $b=1$ and $c=k_1+1$, gives
\begin{align*}
&\hypergeom{2}{1}{1,k_1+k_2+1}{k_1+1}{\frac{1\pm\sqrt{1-x}}{2}}\\&=\frac{2k_1}{k_1+k_2}\frac{1\pm\sqrt{1-x}}{x}+\frac{2k_2}{k_1+k_2}\frac{1\pm\sqrt{1-x}}{x}\hypergeom{2}{1}{1,k_1+k_2}{k_1+1}{\frac{1\pm\sqrt{1-x}}{2}}
\end{align*}
 which is used to eliminate the hypergeometric functions containing $k_1+k_2+1$ and give \eqref{ck1ge}.

For Part 2, use equation~\eqref{recur3} multiplied by $x$ to eliminate the second hypergeometric function on the right hand side of \eqref{ck1ge}.
\end{proof}

The above formulas allow other identities which are interesting in their own right.

\begin{theorem}
The following equalities hold.
\begin{itemize}
\item[1.]
\begin{equation}\label{euler2}
\hypergeom{3}{2}{1,\frac{k_1+k_2}{2}+1,\frac{k_1+k_2+1}{2}}{k_1+1,k_2+1}{x}= \frac{k_1! k_2!}{(k_1+k_2)!}\frac{2^{k_1+k_2}}{x^{k_2}}\frac{(1+\sqrt{1-x})^{k_2-k_1}}{\sqrt{1-x}} \ -
\end{equation}
$$4\frac{k_1 k_2}{(k_1+k_2-1)(k_1+k_2)x}\;\hypergeom{3}{2}{1,-k_2+1,-k_1+1}{\frac{-k_1-k_2}{2}+1,\frac{-k_1-k_2+3}{2}}{\frac{1}{x}}.$$
\item[2.]
\begin{equation}\label{corst}
\hypergeom{3}{2}{1,-k_2+1,-k_1+1}{\frac{-k_1-k_2}{2}+1,\frac{-k_1-k_2+3}{2}}{\frac{1}{x}} = \end{equation}

$$\frac{k_1+k_2-1}{4k_1}\frac{x}{ \sqrt{1-x}} \left[ \hypergeom{2}{1}{1,k_1+k_2}{k_1+1}{ \frac12 -\frac12 \sqrt{1-x}} -\hypergeom{2}{1}{1,k_1+k_2}{k_1+1}{ \frac12 +\frac12 \sqrt{1-x}} \right].  
$$
\end{itemize}
\end{theorem}

\medskip

Note that in the second equality the difference between the two hypergeometric functions on the right hand side gives a polynomial in $\frac{1}{x}$ on the left.

\begin{proof} We note from \eqref{quadra} with $a=\frac{k_1-k_2+1}{2}$ and $b=\frac{k_1-k_2}{2}+1$ that,
$$
\hypergeom{2}{1}{\frac{k_1-k_2}{2}+1,\frac{k_1-k_2+1}{2}}{k_1-k_2+1}{x}=\frac{1}{\sqrt{1-x}}\left(\frac{\sqrt{1-x}+1}{2}\right)^{k_2-k_1}, 
$$
so the first term in the right hand side of \eqref{euler2} can be written as
$$ \frac{2^{2k_2}}{x^{k_2} {k_1 + k_2 \choose k_1}} \hypergeom{2}{1}{\frac{k_1-k_2}{2}+1,\frac{k_1-k_2+1}{2}}{k_1-k_2+1}{x} = 
 \frac{2^{2k_2}}{{k_1 + k_2 \choose k_1}} \sum_{n=0}^\infty \frac{\left(\frac{k_1-k_2}{2}+1\right)_n \left(\frac{k_1-k_2+1}{2}\right)_n}{(k_1-k_2+1)_n} \frac{x^{n-k_2}}{n!} . $$
Split the sum into the terms $n\ge k_2$ and $n\le k_2-1$ and consider the first case,
\begin{align*}
\sum_{n=k_2}^{\infty}\frac{\left(\frac{k_1-k_2}{2}+1\right)_n \left(\frac{k_1-k_2+1}{2}\right)_n}{(k_1-k_2+1)_n(1)_n}x^{n-k_2}&=\sum_{n=0}^{\infty}\frac{\left(\frac{k_1-k_2}{2}+1\right)_{n+k_2} \left(\frac{k_1-k_2+1}{2}\right)_{n+k_2}}{(k_1-k_2+1)_n(1)_{n+k_2}}x^{n}\\&=\frac{\left(\frac{k_1-k_2}{2}+1\right)_{k_2} \left(\frac{k_1-k_2+1}{2}\right)_{k_2}}{(k_1-k_2+1)_{k_2}(1)_{k_2}}\sum_{n=0}^{\infty}\frac{\left(\frac{k_1+k_2}{2}+1\right)_{n} \left(\frac{k_1+k_2+1}{2}\right)_{n}}{(k_1+1)_n(k_2+1)_{n}}x^{n}\\&=\frac{1}{2^{2k_2}}\frac{(k_1+k_2)!}{k_1! k_2!}\hypergeom{3}{2}{1,\frac{k_1+k_2}{2}+1,\frac{k_1+k_2+1}{2}}{k_1+1,k_2+1}{x},
\end{align*}
where the identity $(a)_{n+k_2}=(a)_{k_2}(a+k_2)_n$ has been used four times to obtain the second equality.
This leads to the term on the left hand side of equation~\eqref{euler2}. Consider now the remaining sum,
\begin{align}\label{2ndsum}
&\sum_{n=0}^{k_2-1}\frac{\left(\frac{k_1-k_2}{2}+1\right)_n \left(\frac{k_1-k_2+1}{2}\right)_n}{(k_1-k_2+1)_n(1)_n}x^{n-k_2}=\sum_{m=0}^{k_2-1}\frac{\left(\frac{k_1-k_2}{2}+1\right)_{k_2-m-1} \left(\frac{k_1-k_2+1}{2}\right)_{k_2-m-1}}{(k_1-k_2+1)_{k_2-m-1}(1)_{k_2-m-1}}x^{-m-1}\nonumber\\&=\frac{\left(\frac{k_1-k_2}{2}+1\right)_{k_2-1} \left(\frac{k_1-k_2+1}{2}\right)_{k_2-1}}{(k_1-k_2+1)_{k_2-1}(1)_{k_2-1}}\sum_{m=0}^{k_2-1}\frac{(-k_2+1)_{m}(-k_1+1)_{m}}{(\frac{k_1-k_2}{2}+1)_{m}(\frac{-k_1-k_2+3}{2})_{m}}x^{-m-1},
\end{align}
where the change of variables $n-k_2=-m-1$ was used to obtain the first equality and the identities
$$(1)_{k_2-1-m}=(-1)^m \frac{(k_2-1)!}{(-k_2+1)_m} ,\quad
(k_1-k_2+1)_{k_2-1-m}=(-1)^m\frac{(k_1-k_2+1)_{k_2-1}}{(-k_1+1)_m},
$$
$$
\left(\frac{k_1-k_2}{2}+1\right)_{k_2-1-m}=(-1)^m\frac{\left(\frac{k_1-k_2}{2}+1\right)_{k_2-1}}{\left(\frac{-k_1-k_2}{2}+1\right)_m},
$$ 
and
$$
\left(\frac{k_1-k_2+1}{2}\right)_{k_2-1-m}=(-1)^m\frac{\left(\frac{k_1-k_2+1}{2}\right)_{k_2-1}}{\left(\frac{-k_1-k_2+3}{2}\right)_m},
$$
were used to obtain the second equality.
Again routine manipulations give
$$
\frac{\left(\frac{k_1-k_2}{2}+1\right)_{k_2-1} \left(\frac{k_1-k_2+1}{2}\right)_{k_2-1}}{(k_1-k_2+1)_{k_2-1}(1)_{k_2-1}}=\frac{1}{2^{2k_2-2}}\frac{(k_1+k_2-2)!}{(k_1-1)!(k_2-1)!}
$$
which when used in equation \eqref{2ndsum} leads to the second term on the right hand side of \eqref{euler2}.

To prove \eqref{corst}, simplify take the difference of equations \eqref{ck1ge} and \eqref{euler2}.
\end{proof}

\section{Asymptotics}
\label{sec:Asymptotics}

The asymptotics of the Fourier coefficients of $c_{tk_1, tk_2}$ as $t\to \infty$ when $k_1k_2\le 0$ is clear from Theorem ~\ref{coeffi} above:
$$ \lim_{t\to \infty} \frac{c_{tk_1, tk_2}}{ \left(\frac{r}{2}-\sqrt{\frac{r^2}{4}-1}\right)^{t(|k_1|+|k_2|)}} = \frac{1}{\sqrt{1-\frac{4}{r^2}}} . $$
This exponential decay of $c_{tk_1, tk_2}$ was already established in \cite[Theorem 2.2.1]{GWAnnals} (in particular, equation (2.2.2)), where the asymptotics were given via the common roots of $p$ and $\overleftarrow{p}$ (recall that $\overleftarrow{p} (z_1, z_2  ):= z_1^{n_1} z_2^{n_2} \overline{p} (1/z_1, 1/z_2 )$). In this case, if we solve
$$ 1-\frac{z_1+z_2}{r} = p(z_1,z_2)= 0 = \overleftarrow{p} (z_1,z_2) =z_1z_2-\frac{z_1+z_2}{r},$$ we obtain the two solutions
$$ (z_1,z_2) = \left(\frac{r}{2}-\sqrt{\frac{r^2}{4}-1} , \frac{r}{2}+\sqrt{\frac{r^2}{4}-1}\right) , (z_1,z_2) = \left(\frac{r}{2}+\sqrt{\frac{r^2}{4}-1} , \frac{r}{2}-\sqrt{\frac{r^2}{4}-1}\right). $$
Thus $\frac{r}{2}-\sqrt{\frac{r^2}{4}-1}$ is the component of the intersecting zeros that lies in the open unit disk.


We now turn to the asymptotics for the case when $k_1k_2 >0$.

\begin{theorem}\label{asymt} Let $c_{k_1,k_2}$ be defined as in Theorem \ref{coeffi}.
For $k_1\ge  k_2>0$  the radial asymptotics of these Fourier coefficients are given by
\begin{equation}\label{th4-1}\lim_{t\to\infty}\sqrt{t}\left(\frac{r^{k_1+k_2} k_1^{k_1} k_2^{k_2} }{(k_1+k_2)^{k_1+k_2}} \right)^t  c_{tk_1, tk_2} =\frac{1}{\sqrt{2\pi}} \sqrt{\frac{k_1+k_2}{ k_1 k_2 }}\frac{1}{1-\frac{(k_1+k_2)^2}{r^2k_1k_2}} , \ \ \ \ \ \text{when} \ \frac{(k_1+k_2)^2}{r^2 k_1 k_2 } <1, \end{equation}
\begin{equation}\label{th4-2} \lim_{t\to\infty}\frac{c_{tk_1, tk_2}}{ \left(\frac{r}{2}-\sqrt{\frac{r^2}{4}-1}\right)^{t(k_1-k_2)}}  = \frac{1} {2\sqrt{1-\frac{4}{r^2}}}, \ \ \ \ \ \text{when} \ \ \frac{(k_1+k_2)^2}{r^2 k_1 k_2 } =1 , \end{equation}
and
\begin{equation}\label{th4-3} \lim_{t\to\infty}\frac{c_{tk_1, tk_2}}{\left(\frac{r}{2}-\sqrt{\frac{r^2}{4}-1}\right)^{t(k_1-k_2)}} = \frac{ 1} {\sqrt{1-\frac{4}{r^2}}}, \ \ \ \ \ \text{when} \ \ \frac{(k_1+k_2)^2}{r^2 k_1 k_2 } >1. \end{equation}
In the first limit convergence is uniform on compact subsets of $r>2,\ \frac{(k_1+k_2)^2}{r^2 k_1 k_2 } <1$ while in the third limit the convergence is uniform for on compact subsets of $r>2,\ \frac{(k_1+k_2)^2}{r^2 k_1 k_2 } >1$.
\end{theorem}


We will use the following auxiliary result.

\begin{lem}\label{2f1nasy} With $|x|<1$, $0\le\beta\le1$ and $(1+\beta)|x|<1$, we have the convergence 
\begin{equation}\label{asyn2f1}
\lim_{t\to\infty}\hypergeom{2}{1}{1,t(1+\beta)}{t+1}{x}=\frac{1}{1-(1+\beta) x}.
\end{equation}

 With $|x|<1$, $0<\beta\le1$ and $\frac{1+\beta}{\beta}|x|<1$, we have the convergence 
\begin{equation}\label{asyn2f1}
\lim_{t\to\infty}\hypergeom{2}{1}{1,(1+\beta)t}{\beta t+1}{x}=\frac{1}{1-\frac{1+\beta}{\beta} x}.
\end{equation}
The convergence in each of the results is uniform on compact subsets of their respective regions. 

\end{lem}

\begin{proof} The hypergeometric function is expanded as
$$
\hypergeom{2}{1}{1,t(1+\beta)}{t+1}{x}=\sum_{j=0}^{\infty}\frac{(t(1+\beta))_j}{(t+1)_j}x^j.
$$
The general term in the above series is
$$
\frac{(t(1+\beta))_j}{(t+1)_j}x^j=\frac{(t(1+\beta))_j}{(t+1)_j(1+\beta)^j}((1+\beta)x)^j.
$$
Since for each $i \in \mathbb{N}$, $0\le \beta\le 1$ and $t>0$,
$$
\frac{t(1+\beta)+i}{(1+\beta)(t+i)}=\frac{1+\frac{i}{t(1+\beta)}}{(1+\frac{i}{t})}\le 1
$$
and
$$
\frac{t(1+\beta)}{(1+\beta)(t+j)}\le 1,
$$
the series above is majorized by the geometric series in $|x|(1+\beta)$, which is less than one in the region considered. Thus taking the limit $t\to\infty$ inside the sum gives the result for $0\le \beta\le 1$. 

For the second part note that for each $i \in \mathbb{N}$, $0<\beta\le 1$ and $t>0$,
$$
\frac{(1+\beta)t+i}{\beta t+i+1}=\frac{(1+\beta)}{\beta}\frac{1+\frac{i}{(1+\beta)t}}{1+\frac{i+1}{\beta t}}<\frac{(1+\beta)}{\beta}.
$$  
Thus the series expansion of $\hypergeom{2}{1}{1,(1+\beta)t}{\beta t+1}{x}$ is majorized by the geometric series in $|x|\frac{1+\beta}{\beta}$, which gives the result.
\end{proof}

Theorem \ref{3f22f1} and Lemma \ref{2f1nasy} yield the following.

\begin{cor}\label{3f2asy} With $0<\beta\le 1$, and $0 \le \frac{(1+\beta)^2 x}{4\beta}<1$, uniformly on compact subsets of the above region
\begin{equation}\label{asy3f1}
\lim_{t\to\infty}\hypergeom{3}{2}{1,t\frac{(1+\beta)}{2}+1,t\frac{(1+\beta)}{2}+\frac{1}{2}}{t+1,t\beta+1}{x}=\frac{1}{1-\frac{(1+\beta)^2 x}{4\beta}} .
\end{equation}
\end{cor}

\begin{proof} We use \eqref{ck1ge2} but with  $t=k_1$ and $k_2=\beta t$ so $\beta=\frac{k_2}{k_1}$ considered as continuous parameters. Note that
$
(1+\beta)\frac{1}{2}(1-\sqrt{1-x})<\frac{(1+\beta)}{\beta}\frac{1}{2}(1-\sqrt{1-x})<1
$
 under the hypothesis. Thus Lemma \ref{2f1nasy} is applicable and gives  
\begin{align*}
&\lim_{t\to\infty}\left(\frac{\beta}{(1+\beta )\sqrt{1-x}}\hypergeom{2}{1}{1,(1+\beta)t}{t+1}{ \frac12 -\frac12 \sqrt{1-x}}+ \frac{1}{(1+\beta)\sqrt{1-x}}\hypergeom{2}{1}{1,(1+\beta)t}{\beta t+1}{ \frac12 -\frac12 \sqrt{1-x}}\right)\\&
=\frac{1}{(1+\beta)\sqrt{1-x}}\left(\frac{\beta}{1-(1+\beta)(\frac12 - \frac12\sqrt{1-x})} + \frac{1}{1-\frac{1+\beta}{\beta}(\frac12 - \frac12\sqrt{1-x})}\right)\\&=
\frac{1}{1-\frac{(1+\beta)^2}{4\beta}x}.
\end{align*}
\end{proof}
We will also need the following corollary. 
\begin{cor}\label{3f2cas2asy} With $0\le x<1$, $0\le\beta<1$, and $\frac{(1+\beta)^2 x}{4\beta}>1$, uniformly on compact subsets of the above region
\begin{align*}
&\lim_{t\to\infty}\frac{\beta}{(1+\beta )\sqrt{1-x}}\left(\hypergeom{2}{1}{1,(1+\beta)t}{t+1}{ \frac12 -\frac12 \sqrt{1-x}}- \hypergeom{2}{1}{1,(1+\beta)t}{t+1}{ \frac12 +\frac12 \sqrt{1-x}}\right)\\&=-\frac{1}{\frac{(1+\beta)^2}{4\beta}x-1}.
\end{align*}
\end{cor}
\begin{proof} Since
$
(1+\beta)\frac{1}{2}(1-\sqrt{1-x})<(1+\beta)\frac{1}{2}(1+\sqrt{1-x})<1
$
 under the hypothesis, Lemma \ref{2f1nasy} is applicable and gives  
\begin{align*}
&\lim_{t\to\infty}\frac{\beta}{(1+\beta )\sqrt{1-x}}\left(\hypergeom{2}{1}{1,(1+\beta)t}{t+1}{ \frac12 -\frac12 \sqrt{1-x}}- \hypergeom{2}{1}{1,(1+\beta)t}{t+1}{ \frac12 +\frac12 \sqrt{1-x}}\right)\\&
=\frac{\beta}{(1+\beta)\sqrt{1-x}}\left(\frac{1}{1-(1+\beta)(\frac12 - \frac12\sqrt{1-x})} - \frac{1}{1-(1+\beta)(\frac12 + \frac12\sqrt{1-x})}\right)\\&=
\frac{1}{(1+\beta)\sqrt{1-x}}\left(\frac{-(1+\beta)\sqrt{1-x}}{1-(1+\beta)+\frac{(1+\beta)^2}{4}x}\right)\\&=
-\frac{1}{\frac{(1+\beta)^2}{4\beta}x-1}.
\end{align*}
\end{proof}

\begin{lem}\label{ineq} Suppose that $r>2$, $0<\beta <1$ and $r^2 < \frac{(1+\beta)^2}{\beta}$. Then
$$ \frac{1}{r^{1+\beta}} \frac{(1+\beta)^{1+\beta}}{\beta^\beta} < \left(\frac{2}{r}\right)^{1+\beta} \left(\frac{4}{r^2}\right)^{-\beta} \left(\sqrt{1-\frac{4}{r^2}} +1\right)^{\beta-1} = \left( \frac{r}{2} - \sqrt{\frac{r^2}{4}-1} \right)^{1-\beta}. $$
\end{lem}

\begin{proof}
Let $f(r) = \left(\frac{r}{2}+\frac12 \sqrt{r^2-4}\right)r^\frac{\beta+1}{\beta-1}$. Then
$$f'(r) = \frac{(r+\sqrt{r^2-4})r^{\frac{2}{\beta-1}}\left( (1+\beta)\sqrt{r^2-4}+(\beta-1)r\right)}{\sqrt{r^2-4}(2\beta-2)} . $$
The denominator is negative due to $2\beta-2 <0$. The numerator is negative as $r^2 < \frac{(1+\beta)^2}{\beta}$ implies $(1+\beta)\sqrt{r^2-4}+(\beta-1)r <0$.  Thus $f'(r)>0$, and thus $f(r)$ is increasing. This implies that for $r < \frac{1+\beta}{\sqrt{\beta}}$, we have that $f(r) < f\left(\frac{1+\beta}{\sqrt{\beta}} \right)$. Rewriting the last inequality yields the result.
\end{proof}

We now turn to the proof of the asymptotics of $c_{tk_1, tk_2}$. Our first proof uses the expressions we established.  We will also need Stirling's approximation formula for binomial coefficients:
\begin{align}\label{stir}
\frac{((1+\beta)t)!}{t!(\beta t)!}&=\frac{\Gamma((1+\beta)t+1)}{\Gamma(\beta t+1)\Gamma(t+1)}\nonumber\\&=\sqrt{\frac{1}{2\pi t}}\frac{(1+\beta)^{(1+\beta) t+\frac{1}{2}}}{\beta^{\beta t+\frac{1}{2}}}\bigg(1-\frac{1}{t} \ \frac{1+\beta+\beta^2}{12 \beta (1 + \beta)} +\frac{1}{t^2}\frac{(1 + \beta+\beta^2)^2}{288\beta^2 (1 + \beta)^2}+O\left(\frac{1}{t^3}\right)\bigg).
\end{align}
This can be obtained using the asymptotic expansion of the Gamma function \cite[p. 88]{Olver}.

\medskip

\noindent {\it First proof of Theorem \ref{asymt}.} Let $k_1\ge k_2 \ge 0$. First we consider the case when $\frac{(k_1+k_2)^2}{r^2 k_1 k_2 } >1$. First note that Theorem \ref{coeffi} and part 1 of Theorem \ref{3f22f1} yield
\begin{align}\label{ck1gek2}
&c_{k_1,k_2}=\nonumber\\&\frac{(k_1 +k_2)!}{(k_1)!(k_2)!r^{k_1+k_2}}\frac{k_2}{(k_1+k_2) \sqrt{1-x}} \left[ \hypergeom{2}{1}{1,k_1+k_2}{k_1+1}{ \frac12 -\frac12 \sqrt{1-x}} -\hypergeom{2}{1}{1,k_1+k_2}{k_1+1}{ \frac12 +\frac12 \sqrt{1-x}} \right]\nonumber\\&+\left(\frac{2}{r}\right)^{k_1+k_2}\frac{x^{-k_2}}{\sqrt{1-x}}(\sqrt{1-x}+1)^{k_2-k_1} ,
\end{align}
where $x=(\frac{2}{r})^2$.

From Corollary \ref{3f2cas2asy} and Stirling's approximation~\eqref{stir} for binomial coefficients the first term in \eqref{ck1gek2} has the asymptotic
$$
-\sqrt{\frac{1}{2\pi t}}\frac{1}{r^{(1+\beta)t}}\frac{(1+\beta)^{(1+\beta) t+\frac{1}{2}}}{\beta^{\beta t+\frac{1}{2}}}\frac{1}{\frac{(1+\beta)^2}{4\beta} x-1}.
$$
which decays faster  than the second term due to Lemma \ref{ineq}. Equation \eqref{th4-3} now follows.

Next, we consider the case when $\frac{(k_1+k_2)^2}{r^2 k_1 k_2 } <1$. With $k_1=t,\ k_2=\beta t$,
we have by Theorem 1 and part 2 of Theorem~\ref{3f22f1} part 2,
\begin{align*}
 &c_{k_1,k_2}=\frac{(k_1 +k_2)!}{(k_1)!(k_2)!r^{k_1+k_2}}\times \\&\left( \frac{k_2}{(k_1+k_2)\sqrt{1-x}}\hypergeom{2}{1}{1,k_1+k_2}{k_1+1}{ \frac12 -\frac12 \sqrt{1-x}}+ \frac{k_1}{(k_1+k_2)\sqrt{1-x}}\hypergeom{2}{1}{1,k_1+k_2}{k_2+1}{ \frac12 -\frac12 \sqrt{1-x}}\right)
\end{align*}
Now apply  Corollary~\ref{3f2asy} and Stirling's formula~\eqref{stir}, and we obtain that 
 the asymptotics as $t\to \infty$ of the Fourier coefficients $c_{t,\beta t}$ are given by
\begin{equation}\label{aonecsv}
\lim_{t\to\infty}\frac{\sqrt{t}\ r^{t(1+\beta)}\beta^{\beta t}}{(1+\beta)^{(1+\beta)t}}c_{t,\beta t}=\frac{1}{\sqrt{2\pi}}\frac{(1+\beta)^{\frac{1}{2}}}{\beta^{\frac{1}{2}}}\frac{1}{1-\frac{(1+\beta)^2}{\beta}\left(\frac{1}{r}\right)^2}
\end{equation}
for $\frac{(1+\beta)^2}{\beta}\left(\frac{1}{r}\right)^2<1$. Thus we have obtained Equation \eqref{th4-1}. 


Finally, it remains to prove the case when $\frac{(k_1+k_2)^2}{r^2 k_1 k_2 } =1$. Let  $k_1=t,\ k_2=\beta t$, then we find that 
$x= \frac{4}{r^2}=\frac{4\beta}{(1+\beta)^2}$ gives $\frac12+\frac12\sqrt{1-x} = \frac{1}{1+\beta}$. By \cite[Theorem 2]{ParisIV} (see also \cite[Section 9]{Jones}) we have that
\begin{equation}\label{Par} \hypergeom{2}{1}{1,t(1+\beta)}{t+1}{ \frac{1}{1+\beta}}=\sqrt{\frac{\pi t  (1+\beta)}{2\beta}} \left( 1+O\left(\frac{1}{\sqrt{t}}\right) \right) \end{equation}
(in the notation of \cite{ParisIV} we have $a=0, b=c=1, \epsilon=\beta+1, \lambda = t, x=\frac{1}{1+\beta}$). 
Let us consider the quotient of the third term and the first term of the right hand side of \eqref{ck1ge} (without the sign). Let  $k_1=t,\ k_2=\beta t$, then we find that this quotient equals
$$ \frac{\beta}{1+\beta} \frac{ \hypergeom{2}{1}{1,t(1+\beta)}{t+1}{ \frac{1}{1+\beta}} (4\beta)^{t\beta} {(1+\beta)t \choose t}}{ 2^{(1+\beta)t}\left(\frac{2}{(1+\beta)}\right)^{(\beta -1)t}} .$$
Using \eqref{Par} and Stirling's formula~\eqref{stir} we obtain that this quotient is approximated as $t\to \infty$ by
$$ \frac{\beta}{1+\beta} \sqrt{\frac{\pi t  (1+\beta)}{2\beta}} \frac{\beta^{t\beta} ((1+\beta)t)^{(1+\beta)t} \sqrt{2\pi (1+\beta)t}}{(1+\beta)^{(1-\beta)t}
t^t (\beta t)^{\beta t} \sqrt{2\pi t} \sqrt{2\pi \beta t}} = \frac12.$$
As the second term in the right hand side of \eqref{ck1ge} is dominated by the other two, we now have that for $x=\frac{4k_1k_2}{(k_1+k_2)^2}=\frac{4}{r^2}$,
$$ \lim_{t\to \infty} \frac{\hypergeom{3}{2}{1,\frac{tk_1+tk_2}{2}+1,\frac{tk_1+tk_2+1}{2}}{tk_1+1,tk_2+1}{x}}{\frac12  \frac{2^{tk_1+tk_2} \left(1+\sqrt{1-x}\right)^{tk_2-tk_1}}{x^{tk_2} {tk_1+tk_2 \choose tk_1}\sqrt{1-x}}} = 1, $$
which finishes the proof for the case $\frac{(k_1+k_2)^2}{r^2 k_1 k_2 } =1$. \hfill $\Box$

Next we provide a proof of the asymptotics based on a general theory presented in the book \cite{PW}. This theory concerns the asymptotics of coefficients in the Laurent series of rational functions in several variables.

\noindent {\it Second proof of Theorem \ref{asymt}.} 
We have
\[
f(z_1,z_2)=\frac{z_1z_2}{p(z_1,z_2)\overleftarrow{p}(z_1,z_2)}.
\]
The Fourier coefficients equal the Laurent coefficients of the unique absolutely convergent series expansion in a domain of $\mathbb{C}^2$ containing the points of $\mathbb{T}^2$. 
We need to consider, in addition to the intersection zeros (where $p=0=\overleftarrow{p}$), the so-called smooth points which are solutions to the equations
\begin{equation}\label{ppp} p(z_1,z_2)=0, \ \ k_1 z_2 \frac{\partial p}{\partial z_2}=k_2z_1\frac{\partial p}{\partial z_1} \end{equation} or the equations
\begin{equation}\label{pppp} \overleftarrow{p}(z_1,z_2)=0, \ \ k_1 z_2 \frac{\partial \overleftarrow{p}}{\partial z_2}=k_2z_1\frac{\partial \overleftarrow{p}}{\partial z_1} .\end{equation} Solving \eqref{ppp} gives us
$\left(z_1, z_2\right) = \left(\frac{rk_1}{k_1+k_2}, \frac{rk_2}{k_1+k_2}\right)$. When we solve \eqref{pppp} we find the solutions $(0,0)$ and $({z_1}, {z_2})=\left(\frac{k_1+k_2}{rk_1}, \frac{k_1+k_2}{rk_2}\right)$.
As observed in \cite[Theorem 2.4.4]{Wong}, we know the asymptotics is $O\left(\lvert z_1^{-k_1}z_2^{-k_2}\rvert^t\right)$, where $(z_1,z_2)$ is a smooth point or an intersecting zero. 

Let now $\frac{(k_1+k_2)^2}{r^2 k_1 k_2 } <1$. 
Using \cite[Proposition 5.4]{Melczer}, we see that the point $\left(\frac{rk_1}{k_1+k_2}, \frac{rk_2}{k_1+k_2}\right)$ contributes to the asymptotics. Let $\mathbf{w}=\left(\frac{rk_1}{k_1+k_2}, \frac{rk_2}{k_1+k_2}\right)$ and $\mathbf{x}=(1,1)$. Since the Fourier coefficients are non-negative, the series expansion is combinatorial(see \cite[Definition 5.8]{Melczer}). Therefore, we need to show for $0 < s < 1$,
\[
(z_1(s),z_2(s)) = s\mathbf{w}+(1-s)\mathbf{x} = \left(\frac{rk_1}{k_1+k_2}+1-s, \frac{rk_2}{k_1+k_2}+1-s\right)
\]
cannot be a zero of $p(z_1,z_2)$ or $\overleftarrow{p}(z_1,z_2)$. Indeed, for $p(z_1,z_2)$,
\[
p(z_1(s), z_2(s)) = 1 - \frac{1}{r}\left(\frac{rk_1}{k_1+k_2}+1-s + \frac{rk_2}{k_1+k_2}+1-s\right)
=1-\frac{1}{r}(r+2-2s),
\]
which can only equal to 0 when $ s = 1$. For $\overleftarrow{p}(z_1,z_2)$,
\begin{align*}
\overleftarrow{p}(z_1(s),z_2(s))& =\left(\frac{rk_1}{k_1+k_2}+1-s\right)\left(\frac{rk_2}{k_1+k_2}+1-s\right)-\frac{1}{r}\left(\frac{rk_1}{k_1+k_2}+1-s + \frac{rk_2}{k_1+k_2}+1-s\right)\\
&=\left(\frac{k_1k_2r^2}{(k_1+k_2)^2}-r+1\right)s^2+\left(r-3+\frac{2}{r}\right)s+1 - \frac{2}{r}.
\end{align*}
Since $\frac{k_1k_2r^2}{(k_1+k_2)^2}>1$,
\[
\overleftarrow{p}(z_1(s),z_2(s)) > (2-r)s^2+\left(r-3+\frac{2}{r}\right)s+1 - \frac{2}{r} =:g(s).
\]
Solving the equation $g(s)=0$, we have that $s=-1/r$ and $s = 1$. Therefore, if there are solutions to $\overleftarrow{p}(z_1(s),z_2(s))=0$, we must have that $s < -1/r$ or $s > 1$. Thus, compared to $\left(\frac{rk_1}{k_1+k_2}, \frac{rk_2}{k_1+k_2}\right)$, there is no other singularity that is coordinate-wise closer to a point on $\mathbb{T}^2$.

The rate is given by \cite[Theorem 9.2.7]{PW} (see also \cite[Theorem 9.5.7]{PW}, which is specific to the two variable case), with (in the notation of \cite{PW})


\begin{align*}
f(z_1,z_2)&=\frac{z_1z_2/(z_1z_2-z_1/r-z_2/r)}{(1-z_1/r-z_2/r)}=\frac{G(z_1,z_2)}{H(z_1,z_2)},\\
\frac{-G(z_1,z_2)}{\frac{\partial H}{\partial z_2} (z_1,z_2)}\biggr\rvert_{\left(\frac{k_1r}{k_1+k_2},\frac{k_2r}{k_1+k_2}\right)}&=\frac{1}{\frac{1}{r}\left( 1-\frac{(k_1+k_2)^2}{r^2k_1k_2}\right)},\\
\phi(\theta)&=\log\frac{r-z_1e^{i\theta}}{r-z_1}+i\theta\frac{k_1}{k_2}, \hbox{\rm with} \ z_1=\frac{rk_1}{k_1+k_2},\\
\mathcal{H}(k_1,k_2)&=\phi'' (0) = \frac{k_1(k_1+k_2)}{k_2^2},
\end{align*}
which gives the asymptotic rate
\begin{align*}
\Phi_\mathbf{z}^{(2)}(\mathbf{k}t)&=(2\pi)^{-1/2}
\left(\frac{(k_1+k_2)^{k_1+k_2}}{r^{k_1+k_2} k_1^{k_1} k_2^{k_2} } \right)^t
\left( \frac{k_1(k_1+k_2)}{k_2^2} \right)^{-1/2}
\frac{1}{\frac{1}{r}\left( 1-\frac{(k_1+k_2)^2}{r^2k_1k_2}\right)}
\frac{k_1+k_2}{k_2r}(k_2t)^{-1/2}\\
&=\frac{1}{\sqrt{2\pi t}}\sqrt{\frac{k_1+k_2}{k_1k_2}}\frac{1}{1-\frac{(k_1+k_2)^2}{r^2k_1k_2}}\left(\frac{(k_1+k_2)^{k_1+k_2}}{r^{k_1+k_2} k_1^{k_1} k_2^{k_2} } \right)^t, \mathbf{k}=(k_1,k_2).
\end{align*}

When $\frac{(k_1+k_2)^2}{r^2 k_1 k_2 } > 1$, the contribution of the smooth points becomes less than that of the intersecting zero (use Lemma \ref{ineq} with $\beta = \frac{k_2}{k_1}$).
In order to apply \cite[Theorem 10.3.1]{PW}, we calculate
$$ \det \begin{pmatrix} \frac{z_1\partial p}{\partial z_1 } & \frac{z_2\partial p }{\partial z_2} \cr 
 \frac{z_1\partial \overleftarrow{p}}{\partial z_1 } & \frac{z_2\partial \overleftarrow{p} }{\partial z_2} \end{pmatrix} = 
 \det \begin{pmatrix} -\frac{z_1}{r} & -\frac{z_2}{r} \cr 
 z_1(z_2-\frac{1}{r}) & z_2 (z_1-\frac{1}{r})\end{pmatrix} = \frac{z_1z_2(z_2-z_1)}{r} , $$
 which evaluated at the intersecting zero gives $\frac{\sqrt{r^2-4}}{r} = \sqrt{1 - \frac{4}{r^2}}$. Now 
\cite[Theorem 10.3.1]{PW} yields the asymptotic rate (in their notation)
$$
\Phi_z({\bf k}t)= \frac{1}{\sqrt{1 - \frac{4}{r^2}}} \left(\frac{r}{2}+\sqrt{\frac{r^2}{4}-1}\right)^{-k_1t} \left(\frac{r}{2}-\sqrt{\frac{r^2}{4}-1}\right)^{-k_2t}= \frac{1}{\sqrt{1 - \frac{4}{r^2}}}\left(\frac{r}{2}+\sqrt{\frac{r^2}{4}-1}\right)^{-(k_1-k_2)t}.
$$

For the case when $\frac{(k_1+k_2)^2}{r^2 k_1 k_2 } = 1$, we apply \cite[Corollary 2.4]{PW2010}, which gives that the convergence is at the rate in the previous case multiplied by $\frac12$.
\hfill $\Box$

\medskip

Both approaches to the proof of Theorem \ref{asymt} also allow us to obtain a complete asymptotic series for $c_{tk_1,tk_2}$ as we will explain in the following remarks. 

\begin{rem}\rm
To obtain a complete asymptotic series for $c_{tk_1,tk_2}$ we need an asymptotic series for the $\ _2F_1$'s appearing in Lemma \ref{2f1nasy}. This can be obtained by using the following formulas from \cite[Formulas (2.5) and (3.4)]{Fields}, 
\begin{equation}\label{poh1}
\frac{((1+\beta) t)_j}{((1+\beta)t)^j}=\sum_{k=0}^j\frac{(1-j)_k}{k!} B^{(j)}_k(0)((1+\beta)t)^{-k}=:\sum_{k=0}^{j-1}\alpha_{j,k}t^{-k},
\end{equation}
and for $a>0$ and $at>j$,
\begin{equation}\label{poh2}
\frac{(at)^j}{(at+1)_j}=\frac{(at)^{j+1}}{(at)_{j+1}}=\sum_{k=0}^{\infty}(-1)^k\frac{(j+1)_k}{k!} B^{(-j)}_k (0) (at)^{-k}=:\sum_{k=0}^{\infty}\gamma_{j,k}t^{-k}.
\end{equation}
Here the $B^{(j)}_k(0)$  are the generalized Bernoulli numbers given by
\begin{equation}\label{bern}
\left(\frac{t}{e^t-1}\right)^j=\sum_{k=0}^{\infty}\frac{t^k}{k!} B^{(j)}_k(0), |t|<2\pi .
\end{equation}
The first few are  $B^{(j)}_0(0)=1$, $B^{(j)}_1(0)=-\frac{j}{2}$, and $B^j_2(0)=\frac{j(3j-1)}{12}$.
We now modify the argument given in Theorem~2 and Lemma~2 of \cite{Fields}.
Suppose $0<a\le 1+\beta$, $0\le u\le 1$, and $i\ge 0$. 
Let $$a_i(x)= \frac{1+\frac{ix}{1+\beta}}{1+\frac{(i+1)x}{a}}= 
\frac{a}{i+1} \left( \frac{i}{1+\beta} + \frac{ \frac{i+1}{a}-\frac{i}{1+\beta} }{1+\frac{i+1}{a}x} \right) . 
$$ Then
we have for $0\le i\le k-1$ and $x\ge 0$,
\begin{equation}\label{nderiv}
\left|a^{(n)}_i(x)\right| =  \frac{\left(\frac{i+1}{a}-\frac{i}{1+\beta} \right) n! \left( \frac{i+1}{a}\right)^{n-1}}{\left(1+\frac{i+1}{a} x\right)^{n+1} } \le n!\left(\frac{k}{a}\right)^{n}, n\ge 1.
\end{equation}
Checking it separately, one sees that \eqref{nderiv} also holds for $n=0$.
Next, let
$$
F_k(x)=\prod_{i=0}^{k-1} \frac{1+\frac{ix}{1+\beta}}{1+\frac{(i+1)x}{a}}= a_0(x) a_1(x) \cdots a_{k-1}(x). 
$$
Using Leibniz's product rule and \eqref{nderiv}, we obtain
$$
\max_{x\ge 0}\left|F_k^{(n)}(x)\right|\le n! \frac{k^{2n}}{a^n}.
$$
From Taylor's theorem  we find
\begin{equation}\label{tay}
\prod_{i=0}^{k-1} \frac{1+\frac{i}{(1+\beta)t}}{1+\frac{(i+1)}{at}}=\sum_{i=0}^{n-1}\frac{d_{k,i}}{t^i}+R_{k,n}(t),
\end{equation}
with the bound
$$
|R_n(t)|\le \frac{k^{2n}}{(at)^n}.
$$
Thus, from the calculations above,
\begin{align}\label{asywitherror}
\bigg|\sum_{k=0}^{\infty}\bigg(\frac{a^k((1+\beta)t)_k}{(1+\beta)^k(at+1)_k}-\sum_{i=0}^{n-1}\frac{d_{k,i}}{t^i}\bigg)\left(\frac{1+\beta}{a}x\right)^k\bigg|&\le\frac{1}{a^nt^n}\sum_{k=0}^{\infty}k^{2n}\left|\left(\frac{1+\beta}{a}x \right)^k\right|,
\end{align}
where for each fixed $n$ the last sum converges as $\left|\frac{1+\beta}{a}x\right|<1$.
Examination of Equations~\eqref{poh1} and \eqref{poh2} shows that
$$
d_{k,i}=\sum_{j=0}^{i} \gamma_{k,j}\alpha_{k,i-j}.
$$
The above calculations are examples of the confluence principle, see \cite{Fields}. 

As we are able to bound the right hand side of \eqref{asywitherror}, we arrive at the following statement.
\begin{thm}\label{betterthm}
With $0\le x\le\frac{1+\beta}{a}x<1$ and the notation above, we have
\begin{equation}\label{asyser1}
\left|\hypergeom{2}{1}{1,t(1+\beta)}{at+1}{x}-\sum_{i=0}^{n-1}\sum_{j=0}^i\sum_{k=0}^{\infty}\frac{\gamma_{k,j}\alpha_{k,i-j}}{t^i}\left(\frac{1+\beta}{a}x\right)^k \right|\le\frac{(2n)!}{2}\frac{(1+\beta)x(a+(1+\beta)x)}{a^{n+2}(1-\frac{(1+\beta)x}{a})^{2n+1}}\frac{1}{t^{n}},
\end{equation}
which leads to the asymptotic series
\begin{equation}\label{simser}
\hypergeom{2}{1}{1,t(1+\beta)}{at+1}{x}\sim 
\sum_{i=0}^{\infty}t^{-i}\left(\sum_{j=0}^{i}\sum_{k=0}^{\infty}\gamma_{k,j}\alpha_{k,i-j}\left(\frac{(1+\beta)}{a}x \right)^k\right).
\end{equation}
\end{thm}

\begin{proof} The only thing left to show is that right hand side of \eqref{asywitherror} is bounded above by 
the right hand side of \eqref{asyser1}. We show this by proving the following.

{\bf Claim.} For $y\in[0,1)$ and $n\in {\mathbb N}$, we have $\sum_{k=1}^\infty k^{2n} y^k \le \frac{(2n)!}{2} \frac{y(1+y)}{(1-y)^{2n+1}}$.\footnote{The exact expression for the left hand side is $\sum_{k=1}^\infty k^{2n} y^k =\frac{y(1+y)p_n(y)}{(1-y)^{2n+1}}$, where $p_n(y)= \frac{d^{2n}}{dt^{2n}} \left(\frac{(1-y)^{2n+1}e^t}{y(1+e^t)(1-ye^t)}\right)|_{t=0}$, which is a polynomial of degree $2n-2$. Item A171692 on the On-Line Encyclopedia of Integer Sequences (oeis.org) gives the coefficients of $p_n(y)$.}
\newline
Indeed, we have $\frac{1}{(1-y)^{2n+1}}= \sum_{j=0}^\infty (2n+1)_j \frac{y^j}{j!}$, so that $\frac{1+y}{(1-y)^{2n+1}} = \sum_{j=0}^\infty a_j y^j$, where $a_0=1$ and 
$$ a_j =(2n+1) (2n+2)\cdots (2n+j-1) \left( \frac{1}{(j-1)!} + \frac{1}{j!} (2n+j) \right),\ j\ge 1. $$
Then
$$ \frac{(2n)!}{2}a_j =\frac{(2n+j-1)!}{j!} \frac{2n+2j}{2} = (j+1)(j+2)\cdots (j+2n-1)(n+j) \ge (j+1)^{2n}, n\ge 1. $$
Thus
$$ \frac{(2n)!}{2} \frac{y(1+y)}{(1-y)^{2n+1}} \ge \sum_{j=0}^\infty(j+1)^{2n}y^{j+1} = \sum_{k=1}^\infty k^{2n}y^k. $$
This proves the claim.

Now one applies the above claim to the right hand side of \eqref{asywitherror} to obtain \eqref{asyser1}.
\end{proof}

Writing out the first few terms in \eqref{simser} gives
\begin{align}\label{asyser2f1}
&\hypergeom{2}{1}{1,t(1+\beta)}{at+1}{x}=
\nonumber\\
&=\sum_{j=0}^{\infty}\left(1+\frac{h_1}{t}+\frac{h_2}{t^2}+O\left(\frac{1}{t^3}\right)\right)\left(\frac{1+\beta}{a}x\right)^j,
\end{align}
where
$$
h_1=\gamma_{1,0}\alpha_{1,1}+\gamma_{1,1}\alpha_{1,0}=\frac{j}{2}\left(\frac{(j-1)}{1+\beta}-\frac{j+1}{a}\right)=-\frac{j}{a}+\frac{j(j-1)}{2}\frac{a-\beta-1}{a(1+\beta)}
$$
and
\begin{align*}
h_2&=\frac{(3j+1)j(j+1)(j+2)}{24 a^2}-\frac{j^2(j-1)(j+1)}{4(1+\beta)a}+\frac{(3j-1)j(j-1)(j-2)}{24(1+\beta)^2}\\&=\frac{j}{a^2}+\frac{j(j-1)}{2}\left(\frac{5(1+\beta-a)+2a}{a^2(1+\beta)}\right)+\frac{j(j-1)(j-2)}{6}\frac{(1+\beta-a)(7(1+\beta)-2a)}{a^2(1+\beta)^2}\\&+\frac{j!}{8(j-4)!}\frac{(1+\beta-a)^2}{a^2(1+\beta)^2}.
\end{align*}
Thus, using the geometric series and its first few derivatives, 
\begin{align}\label{tm2}
&\hypergeom{2}{1}{1,t(1+\beta)}{at+1}{x}\nonumber\\&=\frac{1}{y}-\frac{(1+\beta)x}{a^2t}\frac{1}{y^2}\left(1+(\beta+1-a)\frac{x}{ay}\right)\nonumber\\&+\frac{(1+\beta)x}{a^3t^2}\frac{1}{y^2}\bigg(1+(5(1+\beta-a)+2a)\frac{x}{ay}+(1+\beta-a)(7(1+\beta)-2a)\frac{x^2}{a^2y^2}\nonumber\\&+(1+\beta-a)^2(1+\beta)\frac{3x^3}{a^3y^3}\bigg)+O\left(\frac{1}{t^{3}}\right),
\end{align}
where $y=1-\frac{1+\beta}{a}x$.

We now obtain
\begin{align}\label{asympt}
 &\frac{\beta}{(1+\beta)\sqrt{1-x}}\hypergeom{2}{1}{1,(1+\beta)t}{t+1}{ \frac12 -\frac12 \sqrt{1-x}}+ \frac{1}{(1+\beta)\sqrt{1-x}}\hypergeom{2}{1}{1,(1+\beta)t}{\beta t+1}{ \frac12 -\frac12 \sqrt{1-x}}\nonumber\\&=b_0+\frac{b_1}{t}+\frac{b_2}{t^2}+O\left(\frac{1}{t^{3}}\right),
\end{align}
where, with $y_1=1-(1+\beta)(\frac12-\frac12\sqrt{1-x})$ and $y_2=1-\frac{(1+\beta)}{\beta}(\frac12-\frac12\sqrt{1-x})$, and $y_1y_2=1-\frac{(1+\beta)^2}{4\beta}x$,
\begin{equation*}
b_0=\frac{1}{(1+\beta)\sqrt{1-x}}\left(\frac{\beta}{y_1}+\frac{1}{y_2}\right)=\frac{1}{1-\frac{(1+\beta)^2}{4\beta}x},
\end{equation*}
\begin{align*}
b_1&=-\frac{1-\sqrt{1-x}}{2\sqrt{1-x}}\left(\frac{\beta}{y_1^2}+\frac{1}{\beta^2 y_2^2}\right)-\frac{(1-\sqrt{1-x})^2}{4\sqrt{1-x}}\left(\frac{\beta^2}{y_1^3}+\frac{1}{\beta^3y_2^3}\right)\\&=\frac{x}{16\beta^2}\frac{(1+\beta)^3 x-4(1+\beta^3)}{\left(1-\frac{(1+\beta)^2}{4\beta}x \right)^3},
\end{align*}
and with the help of Mathematica \cite{wolfram},
\begin{align*}
b_2&=\frac{x}{(4\beta)^4(1-\frac{(1+\beta)^2}{4\beta}x)^5}\bigg(64\beta(1+\beta^4)\\&+16(1+\beta)^2(1+(-3+\beta)\beta)(2+\beta(2\beta-1))x-20(1-\beta)^2(1+\beta)^4 x^2+(1+\beta)^6 x^3\bigg).
\end{align*}
Note that Theorem~\ref{betterthm} yields that the error term $O(1/t^3)$ in \eqref{asympt} is bounded by 
\begin{equation}\label{errterm}
\left| O\left(\frac{1}{t^3}\right)\right|\le\frac{1}{t^3}\frac{6!(1-\sqrt{1-x})}{4\sqrt{1-x}}\frac{\bigg((1+(1+\beta)\frac{(1-\sqrt{1-x})}{2})\beta^6y_2^7+(\beta+(1+\beta)\frac{(1-\sqrt{1-x})}{2})y_1^7\bigg)}{\beta^5(1-\frac{(1+\beta)^2}{4\beta}x)^7}.
\end{equation}
Substituting $x=\frac{4}{r^2}$ in \eqref{asympt} and combining it with Stirling's formula~\eqref{stir}, we find
\begin{align}\label{ctbt}
&\sqrt{2\pi t}\frac{\beta^{\beta t+\frac{1}{2}}}{(1+\beta)^{(1+\beta) t+\frac{1}{2}}}r^{(1+\beta)t}c_{t,\beta t}=\nonumber\\
&b_0+\frac{1}{t}\left(b_1-\frac{(1+ \beta+ \beta^2)b_0}{12 \beta (1 + \beta)}\right)+\frac{1}{t^2}\left(b_2-\frac{(1+\beta+\beta^2)b_1}{12\beta(1 + \beta)}+\frac{(1 + \beta+\beta^2)^2b_0}{288\beta^2 (1 +\beta)^2}\right) +O\left(\frac{1}{t^3}\right).
\end{align}
An estimate on the error term $O(1/t^3)$ in \eqref{ctbt} can be obtained by using Equation~\eqref{errterm} and the error in Stirling's approximation for combinations. 

An alternative to the above calculations can be obtained by beginning with the Euler integral representation for the above ${}_2F_1$s. A complete asymptotic expansion using this integral representation has been presented in \cite[Formula (3.13)]{ParisI} and \cite[Formula 1.2] {ParisIV} for the case $\epsilon=(1+\beta)$ or $\epsilon=\frac{1+\beta}{a}$, and $x\epsilon<1$. 
\hfill $\square$
\end{rem}

\begin{rem}\label{melczer}\rm
The results \cite[Theorem 9.2.7]{PW} and \cite[Theorem 5.3]{Melczer} also give a complete asymptotic series. Consequently, one can use these results to obtain additional terms in the expansion. For instance, in the case when $\frac{(k_1+k_2)^2}{r^2k_1k_2} <1$, computation yields\footnote{The codes that were used are due to Stephen Melczer and are available via {\sf melczer.ca/textbook/}.}
$$ c_{tk_1,tk_2} = \frac{1}{\sqrt{2\pi t}} \frac{r^2\sqrt{k_1k_2}\sqrt{{k_1+k_2}}}{r^2k_1k_2-(k_1+k_2)^2}\left( \frac{(k_1+k_2)^{k_1+k_2}}{r^{k_1+k_2} k_1^{k_1} k_2^{k_2} } \right)^t \left( 1- \frac{H}{t} + O \left(\frac{1}{t^2}\right)  \right) ,  $$ where  
$$ H=\frac{k_1^2k_2^2(k_1^2+k_1k_2 + k_2^2) r^4 + 2k_1k_2(k_1+k_2)^2(5k_1^2-7k_1k_2+5k_2^2)r^2+(k_1+k_2)^4(k_1^2-11k_1k_2+k_2^2)}{12 (r^2k_1k_2-(k_1+k_2)^2)^2 k_1k_2(k_1+k_2)}. $$
The expressions for following terms can also be computed in this way. For instance, if we specify the values $k_1 =2 , k_2=1, r=3$ (to keep the formulas contained), we obtain
\begin{equation}\label{melc} 
c_{2t,t}= \frac{\sqrt{3}\left( \frac14 \right)^t}{\sqrt{\pi t}} \left( 1-\frac{55}{2^3 3^2t} + \frac{26065}{2^7 3^4 t^2} -
\frac{32881015}{2^{10} 3^7 t^3} + \frac{78037754977}{2^{15} 3^9 t^4} + O \left( \frac{1}{t^5} \right) \right) . 
\end{equation}
In order match \eqref{melc} with \eqref{ctbt} one needs to take $\beta = \frac12$, $r=3$, and do a parameter change $t \to 2t$. \hfill $\square$
\end{rem}

\bigskip

The asymptotic result in Theorem \ref{asymt} for the case when $\frac{(k_1+k_2)^2}{r^2 k_1 k_2 } =1$ may be written in the following way, which is of potential interest in probability.

\begin{cor} With $k_1>k_2>0$ and $p=\frac{k_1}{k_1+k_2}$ (and thus $1-p=\frac{k_2}{k_1+k_2}$), we have
\begin{equation}\label{cutoff3} \lim_{t \to \infty} \left( \sum_{i=0}^\infty {tk_1+tk_2 +2i \choose tk_1+i } p^{tk_1+i} (1-p)^{tk_2+i} \right) =  \frac{\frac14}{p-\frac12} .\end{equation}
\end{cor}

\begin{proof}
Recall that \eqref{c000} gives us
\begin{equation}\label{c0000}
c_{tk_1,tk_2}=\sum_{i=0}^{\infty}\frac{1}{r^{2i+tk_1+tk_2}}\frac{(tk_1+i+1)_{tk_2}(tk_1+i+tk_2+1)_i}{(1)_{tk_2}(tk_2+1)_i}.
\end{equation}
When $\frac{(k_1+k_2)^2}{r^2 k_1 k_2 } =1$ we have $$ r=\frac{k_1+k_2}{\sqrt{k_1 k_2}},  $$
and thus
$$ \frac{r^2}{4} - 1 = \frac{(k_1+k_2)^2 }{4k_1k_2} - 1 = \frac{k_1^2-2k_1k_2 +k_2^2}{4k_1k_2} = \frac{(k_1-k_2)^2}{4k_1k_2}. $$
Also,
$$ 1-\frac{4}{r^2} = \frac{(k_1+k_2)^2 }{(k_1+k_2)^2 }- \frac{4k_1k_2}{(k_1+k_2)^2 }=\frac{(k_1-k_2)^2 }{(k_1+k_2)^2 }, $$
and
$$ \frac{r}{2}-\sqrt{\frac{r^2}{4}-1}=\frac{k_1+k_2}{2\sqrt{k_1 k_2}}-\frac{k_1-k_2}{2\sqrt{k_1 k_2}}=\sqrt{\frac{k_2}{k_1}}.$$

In Theorem \ref{asymt} we have proven that 
\begin{equation}\label{cutoff} \lim_{t\to\infty}\frac{c_{tk_1, tk_2}}{\left(\frac{r}{2}-\sqrt{\frac{r^2}{4}-1}\right)^{t(k_1-k_2)}} = \frac12 \frac{ 1} {\sqrt{1-\frac{4}{r^2}}}.  \end{equation}

It is easy to check that
$$ \frac{(tk_1+i+1)_{tk_2}(tk_1+i+tk_2+1)_i}{(1)_{tk_2}(tk_2+1)_i} = {tk_1+tk_2 +2i \choose tk_1+i }. $$
Also,
$$ \frac{1}{r^{2i+tk_1+tk_2}} \frac{1}{ \left(\frac{r}{2}-\sqrt{\frac{r^2}{4}-1}\right)^{t(k_1-k_2)}} = 
\left( \frac{\sqrt{k_1 k_2}}{k_1+k_2}\right)^{tk_1+tk_2 + 2i} \left( \sqrt{\frac{k_1}{k_2}} \right)^{tk_1-tk_2}= \frac{k_1^{tk_1+i} k_2^{tk_2+i}}{(k_1+k_2)^{tk_1+tk_2+2i}} .$$

Thus \eqref{cutoff} is equivalent to 
\begin{equation}\label{cutoff2} \lim_{t \to \infty} \left( \sum_{i=0}^\infty {tk_1+tk_2 +2i \choose tk_1+i } \frac{k_1^{tk_1+i} k_2^{tk_2+i}}{(k_1+k_2)^{tk_1+tk_2+2i}} \right) = \frac12 \frac{k_1+k_2 }{k_1-k_2}. \end{equation}

%
If we let $p=\frac{k_1}{k_1+k_2}$, then $1-p= \frac{k_2}{k_1+k_2}$, and $2 \frac{k_1+k_2 }{k_1-k_2} =\frac{1}{p-\frac12}$, this transforms \eqref{cutoff2} into \eqref{cutoff3}.
\end{proof}

\section{Orthogonal Polynomials}
\label{sec:Ortho}
In this section we consider the bivariate orthonormal polynomials associated with the above Fourier coefficients where the monomials are ordered lexicographically. The orthogonal polynomials constructed with this ordering and their characterizing properties were studied in \cite{DGK,GWSIAM}.
We introduce the symmetric sesquilinear form (or Hermitian form)
$$ \langle g,h \rangle = \frac{1}{4\pi^2} \int_0^{2\pi} \int_0^{2\pi} \frac{g(e^{i\theta} , e^{i\phi}) \overline{h(e^{i\theta} , e^{i\phi})}}{|p(e^{i\theta} , e^{i\phi})|^2} d\theta \ d\phi  =  \frac{1}{4\pi^2} \int_{\mathbb T}  \int_{\mathbb T}  \frac{g (z_1,z_2)  \overline{ h ({z_1} , {z_2} )}}{p(z_1,z_2)\bar p(\frac{1}{z_1} , \frac{1}{z_2} )  } \frac{dz_1}{iz_1} \frac{dz_2}{iz_2} $$
for $g$ and $h$ square integrable functions on the bidisk ${\mathbb T}^2$. Denote ${\mathbb Z}_+ = \{ 0, 1, 2, \ldots \}$.

\begin{theorem}\label{ortho} Let 
$$ a=\frac{r}{2}-\sqrt{\frac{r^2}{4}-1}.$$ Then the two variable polynomials
$$ \sqrt[4]{1-\frac{4}{r^2}} , $$ $$ z_1^{k} \frac{\sqrt[4]{1-\frac{4}{r^2}}}{\sqrt{1-a^2}}(z_1-a), k\in {\mathbb Z}_+, $$  $$ z_2^{k} \frac{\sqrt[4]{1-\frac{4}{r^2}}}{\sqrt{1-a^2}}(z_2-a), k\in {\mathbb Z}_+, $$
$$ z_1^{k_1}z_2^{k_2}\left(z_1z_2 - \frac{z_1}{r} - \frac{z_2}{r}\right), (k_1,k_2) \in {\mathbb Z}_+^2, $$
form a complete set of orthonormal polynomials with respect to $\langle \cdot , \cdot \rangle$.
\end{theorem}

Notice that by \cite[Proposition 2.1.2]{GWAnnals} we have 
$$ \left( \sum_{k\in {\mathbb Z}} c_{k,0}z_1^k \right)^{-1} = \left(1-\frac{z_1}{r}\right)\left(1-\frac{1}{z_1r}\right)-\frac{1}{r^2} = 1-\frac{z_1}{r}-\frac{1}{z_1r}, $$
which has roots $a$ and $\frac{1}{a}$. This explains why the orthogonal polynomials that only have the variable $z_1$, have $a$ as a root. The same is true for the orthogonal polynomials that only have the variable $z_2$.

\begin{proof} By Theorem \ref{coeffi},
$$ \langle z_1^{k_1} z_2^{k_2} , z_1^{l_1} z_2^{l_2} \rangle =  \langle z_1^{k_1-l_1} z_2^{k_2-l_2} , 1 \rangle = c_{k_1-l_1,k_2-l_2} = \frac{a^{|k_1-l_1|+|k_2+l_2|}}{\sqrt{1-\frac{4}{r^2}}} , \ (k_1-l_1)(k_2-l_2)\le 0. $$
Using this it is straightforward to see that
$$ z_1^k (z_1-a) \perp {\rm Span} \{ z_1^{l_1} z_2^{l_2} : l_1=0,1,\ldots , k,\ l_2=0, 1,2, \ldots \}, k \ge 0,$$ 
and 
$$ z_2^k (z_2-a) \perp {\rm Span} \{ z_1^{l_1} z_2^{l_2} : l_1=0,1, 2\ldots ,\ l_2=0, 1, \ldots, k \}  , k \ge 0. $$ 
In addition,
$$ \langle 1,1 \rangle = \frac{1}{\sqrt{1-\frac{4}{r^2}}} , $$ $$\langle z_1^k(z_1-a),z_1^k( z_1-a )\rangle = \langle z_1-a, z_1-a \rangle = c_{0,0} -2ac_{1,0} + a^2 c_{0,0} = \frac{1-a^2}{\sqrt{1-\frac{4}{r^2}}} , $$
(and the same equality holds when $z_1$ is replaced by $z_2$), explaining the normalization factors.

Next, let $q(z_1,z_2) = z_1z_2 - \frac{z_1}{r} - \frac{z_2}{r}$. Note that $\bar p(\frac{1}{z_1} , \frac{1}{z_2} ) = \frac{1}{z_1z_2} q(z_1,z_2)$. Thus, 
$$ \langle z_1^{k_1} z_2^{k_2} q (z_1,z_2) , 1 \rangle = \frac{1}{4\pi^2} \int_{\mathbb T}  \int_{\mathbb T}  \frac{z_1^{k_1} z_2^{k_2} q (z_1,z_2) }{p(z_1,z_2)\bar p(\frac{1}{z_1} , \frac{1}{z_2} )  } \frac{dz_1}{iz_1} \frac{dz_2}{iz_2} = 
\frac{1}{4\pi^2} \int_{\mathbb T}  \int_{\mathbb T}  \frac{z_1^{k_1+1} z_2^{k_2+1}  }{p(z_1,z_2) } \frac{dz_1}{iz_1} \frac{dz_2}{iz_2}
=0 $$
when $k_1, k_2 \ge -1$ and $(k_1,k_2) \neq (-1,-1)$. Here we used that $\frac{z_1^{k_1+1} z_2^{k_2+1}}{p(z_1, z_2)}$ is analytic in the closed bidisk, so that the integral equals its value at 0 (due to Cauchy's integral formula). Since 
$$ \langle z_1^{k_1} z_2^{k_2} q (z_1,z_2) ,  z_1^{l_1} z_2^{l_2}\rangle =  \langle z_1^{k_1-l_1} z_2^{k_2-l_2} q (z_1,z_2) , 1 \rangle , $$
we obtain that 
\begin{equation}\label{or} z_1^{k_1} z_2^{k_2} q (z_1,z_2) \perp {\rm Span} \{ z_1^{l_1} z_2^{l_2} : l_1 = 0, \ldots , k_1+1, l_2 = 0, \ldots , k_2 +1, (l_1,l_2) \neq (k_1+1, k_2+1 ) \} . \end{equation}
Finally, 
$$ \langle z_1^{k_1}z_2^{k_2} q (z_1,z_2) ,  z_1^{k_1}z_2^{k_2}q (z_1,z_2) \rangle = \langle q (z_1,z_2) ,  q (z_1,z_2) \rangle=\langle q (z_1,z_2) ,  z_1 z_2\rangle = $$
$$ \frac{1}{4\pi^2} \int_{\mathbb T}  \int_{\mathbb T}  \frac{z_1^{-1} z_2^{-1} q(z_1, z_2)  }{p(z_1,z_2) p(\frac{1}{z_1} , \frac{1}{z_2} )} \frac{dz_1}{iz_1} \frac{dz_2}{iz_2} = \frac{1}{4\pi^2} \int_{\mathbb T}  \int_{\mathbb T}  \frac{1  }{p(z_1,z_2) } \frac{dz_1}{iz_1} \frac{dz_2}{iz_2} =1 , $$ again by Cauchy's integral formula. Note that in the second equality we used \eqref{or}. The characterization of orthonormal polynomials in this ordering (see \cite{DGK,GWSIAM}) now
establishes the proof.
\end{proof}

\begin{acknowledgements}
We wish to thank Professors Robin Pemantle, Stephen Melczer, and Richard Paris for many useful discussions. In addition, we very much appreciate Stephen Melczer generously sharing his Maple and Sage code with us. We also greatly appreciate the detailed and constructive feedback from the referees.
\end{acknowledgements}

%
%



\end{document}